\documentclass[a4paper,11pt]{amsart}

\usepackage[T1]{fontenc}
\usepackage[utf8]{inputenc}
\usepackage{lmodern}
\usepackage[a4paper,margin=1.05in,footskip=0.35in]{geometry}
\usepackage{microtype}
\usepackage{amsmath,amssymb,amsthm,mathtools}
\usepackage{mathrsfs}
\usepackage[cmyk]{xcolor}
\usepackage{enumitem}
\usepackage[colorlinks=true,linkcolor=blue,citecolor=blue,urlcolor=blue]{hyperref}

\allowdisplaybreaks
\setlist{itemsep=0.2em,topsep=0.35em}

\theoremstyle{plain}

\newtheorem{theorem}{Theorem}[section]
\newtheorem{proposition}[theorem]{Proposition}
\newtheorem{lemma}[theorem]{Lemma}
\newtheorem{definition}[theorem]{Definition}
\newtheorem{assumption}[theorem]{Assumption}
\newtheorem{corollary}[theorem]{Corollary}
\numberwithin{equation}{section}

\theoremstyle{remark}
\newtheorem{remark}[theorem]{Remark}

\newcommand{\M}{\mathcal M}

\newcommand{\NN}{\mathcal N}
\newcommand{\WW}{\mathcal W}
\newcommand{\GG}{\mathcal G}
\newcommand{\N}{\mathbb N}
\newcommand{\II}{\mathcal I}
\newcommand{\BB}{\mathcal B}
\renewcommand{\SS}{\mathcal S}
\renewcommand{\AA}{\mathcal A}
\newcommand{\R}{\mathbb R}
\newcommand{\C}{\mathbb C}
\newcommand{\tr}{\operatorname{tr}}
\title{New matrix-weighted norm inequalities and extrapolation}
\author[]{Guixiang Hong}
\address{
Institute for Advanced Study in Mathematics\\
Harbin Institute of Technology\\
Harbin
150001\\
China}
\email{gxhong@hit.edu.cn}

\author[]{Pengfei Yin}
\address{
Institute for Advanced Study in Mathematics\\
Harbin Institute of Technology\\
Harbin
150001\\
China}
\email{p.yin@stu.hit.edu.cn}

\thanks{}

\subjclass[2020]{Primary 46L52, 42B25; Secondary 46L51, 42B20}

\keywords{Matrix weights, noncommutative maximal inequalities, factorization of weights, Rubio de Francia extrapolation, dimension-free}
\begin{document}

\begin{abstract}
In this paper, we propose a new type of matrix-weighted theory and show the corresponding Rubio de Francia extrapolation theorems by means of noncommutative maximal inequalities. We introduce a new matrix $A_p$ condition following Volberg's geometric formulation, and characterize it via a noncommutative Hardy--Littlewood maximal inequality. We then establish a noncommutative analogue of Rubio de Francia's abstract extrapolation theorem for trace-finite von Neumann algebras. The argument extends to $L_\infty(\mathbb R)\bar\otimes M_d(\mathbb C)$ through pointwise matrix operations.  This yields the desired extrapolation results for individual operators and maximal families, with applications to matrix-weighted boundedness of Calder\'on--Zygmund operators.
\end{abstract}

\maketitle
\section{Introduction}\label{sec:intro}

Weighted norm inequalities are a central part of modern
harmonic analysis.  On the real line, we say a function $w\colon \R\to \R_+$ is a scalar weight if it is positive, measurable and locally integrable. Given $1\leq p<\infty$ and a weight $w$, the weighted $L_p$-norm is
\[
    \| f\|_{L_p(w)}:=\Big(\int_\R |f(t)|^pw(t)d t\Big)^{1/p}.
\]
For $1<p<\infty$, the Muckenhoupt class $A_p(\R)$ is characterized by
\[
 [w]_{A_p}
 =\sup_{I\subset\R}\langle w\rangle_I
       \langle w^{-p'/p}\rangle_I^{p/p'}<\infty,
 \qquad
 \langle w\rangle_I=\frac1{|I|}\int_Iw(t)\,dt,
\]
where the supremum is taken over bounded intervals. Muckenhoupt's theorem in \cite{Mucken1972} identifies this condition with the boundedness of the Hardy--Littlewood maximal operator 
\[
    Mf(t):=\sup_{I\ni t}\frac{1}{|I|}\int_I |f(s)|ds
\] 
on $L_p(w)$.  The same class also governs the weighted boundedness of the
Hilbert transform
\[
    Hf(t):=\mathrm{p.v.}\frac{1}{\pi}\int_\R \frac{f(s)}{t-s}ds
\] 
and, more generally, of Calder\'on--Zygmund
operators \cite{HMW1973,CoiFeffer1974}.

The extrapolation of Rubio de Francia in \cite{Rubio1984} explains why weighted norm inequalities persist
between exponents.  In one of its standard forms, if a sublinear operator $S$ satisfies
\[
 \|Sf\|_{L_r(w)}
 \le \Phi([w]_{A_r})\|f\|_{L_r(w)}
\]
for all $w\in A_r(\R)$ at one exponent $1<r<\infty$, then corresponding
estimates hold at every $1<p<\infty$ and $w\in A_p(\R)$.
The proof is based on the Rubio de Francia
iteration algorithm, which exploits the Hardy-Littlewood maximal operator to build a majorant with controlled $A_p$ characteristics.

Matrix-weighted theory differs substantially from its scalar counterpart. We say $w:\R\to M_d(\C)$ is a matrix weight if it is measurable, positive definite almost everywhere,
and locally integrable. Given $f\colon \R \to \ell_2^d$, the usual Hilbert-valued realization is
\[
 \|f\|_{L_p(w;\ell_2^d)}
 =\left(\int_\R
    \|w(t)^{1/p}f(t)\|_{\ell_2^d}^{p}\,dt\right)^{1/p}.
\]
For $d=1$, this reduces to the scalar weighted norm. For $d>1$, Treil and Volberg proved the matrix $A_2$ characterization of the Hilbert
transform \cite{TreilVolberg1997} but with the sharp $A_2$-bound different from the classical one \cite{NPTV2017,Domelevo2024}, and Volberg's geometric formulation of
$A_p(\R;M_d(\C))$ describes matrix weights through averaged norm duality
\cite{Volberg1997}. Further characterizations and maximal function approaches were developed by Nazarov-Treil, Christ-Goldberg and Goldberg in
\cite{NazarovTreil1996,ChristGoldberg2001,Goldberg2003}.
In particular, Goldberg's Hilbert-valued maximal theorem can be stated as follows: the maximal operator
\begin{equation}\label{equ:goldberg maximal}
    M_wf(t):=\sup_{I\ni t}\frac{1}{|I|}\int_I \| w^{1/p}(t)w^{-1/p}(s)f(s)\|_{\ell_2^d}ds
\end{equation}
is bounded from $L_p(\R;\ell_2^d)$ to $L_p(\R)$ precisely for matrix $A_p(\R;M_d(\C))$ weights, with constants depending on the corresponding matrix $A_p$ characteristic \cite{IsraMoen2019}. These results provide a robust matrix-weighted theory, but the scalar Rubio
de Francia algorithm does not transfer directly.

The obstruction is simple to state. In the scalar theory, Rubio de Francia’s algorithm
repeatedly applies a positive maximal operator and uses pointwise suprema. In the matrix-valued setting there is no canonical vector-valued supremum with the same input and output
structure. Goldberg’s maximal function, for instance, characterizes matrix $A_p$ weights, but it
maps vector-valued functions to scalar-valued functions. This loss of the original output space
prevents a direct iteration argument.
One recent way around this obstruction was proposed by Bownik and Cruz-Uribe, who
use convex set-valued analysis and a corresponding convex set-valued maximal function to
formulate matrix-weighted extrapolation \cite{BownikCruzUribe2025}. This approach preserves an input-output structure
at the level of convex bodies.

\bigskip

In the present paper, we try to address this problem from the viewpoint of noncommutative analysis. 
More precisely, the operator space $L_p(\M;\ell_\infty)$ playing the role of the space of families of functions with the related maximal function in $L_p$ motivates us to search for an extrapolation argument in the noncommutative setting.
The definitions of noncommutative $L_p$-spaces and  $L_p(\M;\ell_\infty)$ will be recalled in Section~\ref{sec:preliminary}. The latter is one of the basic vector-valued noncommutative $L_p$-spaces introduced by Pisier \cite{Pisier1998}  and developed by Junge in his proof of the noncommutative Doob
maximal inequality \cite{Junge2002}. It has since become an essential tool in noncommutative martingale and ergodic theory, including the noncommutative Dunford-Schwartz maximal ergodic theorem of Junge and Xu \cite{JungeXu2007} and the noncommutative Hardy-Littlewood maximal inequality by Mei \cite{Mei2007}. More recently, noncommutative maximal norms have been used to prove maximal ergodic inequalities under doubling conditions \cite{HongLiaoWang2021} and pointwise convergence results for noncommutative Fourier series \cite{HongWangWang2024}.

\subsection{The noncommutative weighted spaces}
For any von Neumann algebra $\M$ and $0<p<\infty$, there is an intrinsic Haagerup $L_p$ space associated with $\M$ \cite{Haagerup1979}. In the semifinite tracial setting,  this construction is compatible with the usual noncommutative $L_p$ spaces and with Kosaki's interpolation approach \cite{Kosaki1984,PisierXu2003}. The symmetric realization has proved particularly useful in reduction arguments for noncommutative $L_p$ spaces \cite{HaagerupJungeXu2010}.

Motivated by this point of view, let $(\M,\tau)$ be a von Neumann algebra of finite trace, by which we mean that $\tau(1)<\infty$. For $0<p<\infty$, and a full-support $w\in L_0^+(\M)$, one defines the symmetric weighted realization
\[
 \|x\|_{L_p(\M,w)}
 =\|w^{1/2p}xw^{1/2p}\|_{L_p(\M)}.
\]
We restrict to trace-finite von Neumann algebras so that full-support positive weights, their inverses, and the products all belong to the same algebra of measurable operators, and thus we may define the weighted space $L_p(\mathcal M,w)$ as  a subset of $L_0(\mathcal M)$. In the general semifinite setting, all this is no longer automatic, and we have no idea how to define a suitable class of weights stable under the various operations.

But an exceptional case that we will take care of in this paper is
\[
 \NN=L_\infty(\R)\bar\otimes M_d(\C),
\]
with the tensor trace that is semifinite.  A matrix weight is, throughout this
paper, a strongly measurable, almost everywhere positive-definite, and
locally integrable function $w:\R\to M_d(\C)$.  On this algebra all
weighted products can be understood pointwise in the matrix fibers.  Thus one may well define
\begin{equation}\label{eq:intro-global-weighted-space}
 \|f\|_{L_p(\NN,w)}
 =\left(\int_\R
   \|w(t)^{1/2p}f(t)w(t)^{1/2p}\|_{S_p^d}^{p}\,dt\right)^{1/p},
\end{equation}
and the resulting weighted space $L_p(\mathcal N,w)$ as a subset of strongly measurable functions. This two-sided space differs from the classical Hilbert-valued
matrix-weighted space above.

\subsection{The maximal inequality characterization}
Our first result identifies the correct matrix $A_p$ class for the validity of a noncommutative Hardy-Littlewood maximal inequality. The definition of
$A_p(\NN)$ used below is given in Section \ref{sec:new Ap}, which is inspired by  Volberg’s averaged-norm formulation:
the matrix weight $w$ induces pointwise norms on $M_d(\C)$, and $A_p(\NN)$ is expressed as a uniform
comparison between averaged dual norms.

Let $\II$ be the family of bounded intervals in $\R$ and for each $I\in \II$, define the averaging operator
\[
 A_If=\left(\frac1{|I|}\int_If(s)\,ds\right)\chi_I.
\]
The weighted noncommutative maximal norm is defined as follows
\[
        \| (A_If)_{I\in\II}\|_{L_p(\NN,w;\ell_\infty)}:=\big\| {\sup_{I\in \II}}^+ A_If\big\|_{L_p(\NN,w)}:=\big\| {\sup_{I\in\II}}^+ w^{1/2p}(A_If)w^{1/2p}\big\|_{L_p(\NN)}.
\]
Let $\SS_\NN^c$ be the subset of bounded operators with compact support on the real line.
Our first main result is a maximal characterization of our new matrix weight class $A_p(\NN)$.
{\begin{theorem}\label{thm:intro-maximal}
Let $1<p<\infty$ and $w$ be a matrix weight.  Then $w\in A_p(\NN)$ if and only if
\begin{equation}\label{equ:maxiaml-intro}
    \big\| {\sup_{I\in\II}}^+ A_If\big\|_{L_p(\NN,w)}\leq C\| f\|_{L_p(\NN,w)},\qquad f\in \SS_\NN^c.
\end{equation}
Moreover, the optimal constant in \eqref{equ:maxiaml-intro} denoted by $C_{p,d}(w)$ satisfies
\begin{equation}\label{eq:intro-maximal-quantitative}
 [w]_{A_p(\NN)}^{1/p}
 \le C_{p,d}(w)
 \le \WW_{p,d}[w]_{A_p(\NN)}^{p'/p},
\end{equation}
with the least upper bound $\WW_{p,d}\leq d^{\frac{1}{p}+(1+p')|\frac{1}{p}-\frac{1}{2}|}\mathcal{G}_{p,d^2}$ where $\mathcal{G}_{p,d^2}$ is the constant from Goldberg's maximal inequality \eqref{equ:Goldberg quantitative}.
\end{theorem}}

The {\it if} part is easy by direct computation. For the {\it only if} part, the first observation is that one may relate $w$ to a Volberg matrix weight $\mathbb W$ acting on the Hilbert space $S_2^d$ in the following way
\[
 \mathbb W(t)(A)=w(t)^{1/2}Aw(t)^{1/2};
\]
see Lemma \ref{lem:compare N and ell_2^d2}. The second observation is that the noncommutative maximal norm is comparable to the vector-valued maximal norm for finite-dimensional matrix-valued functions; see Lemma \ref{lem:positive-max-to-classical}.
The desired noncommutative weighted
maximal inequality then follows from the argument of 
\cite{IsraMoen2019}.

\subsection{The extrapolation theorem}
Our second result is a matrix-valued version of Rubio de Francia's extrapolation theorem. It is formulated in the framework of noncommutative analysis.
\begin{theorem}\label{thm:intro-extrapolation}
Let
$(T_\alpha)_{\alpha\in\AA}$ be linear maps from $\SS_\NN^c$ to
the strongly measurable $M_d(\C)$-valued functions. Suppose, for some $1<r<\infty$ and an increasing function
$\Phi$, one has
\begin{equation}\label{eq:matrix-maximal-extrapolation-hypothesis}
 \big\|{\sup_{\alpha\in\AA}}^+T_\alpha(f)\big\|_{L_r(\NN,w)}
 \le \Phi([w]_{A_r(\NN)})\|f\|_{L_r(\NN,w)}
\end{equation}
for every $w\in A_r(\NN)$ and $f\in \SS_\NN^c$.  Then, for every
$1<p<\infty$, there exists another increasing function $\Phi_{p,r,d}$ only depending on $p,r,d,\Phi$ such that for any $w\in A_p(\NN)$, the map
$f\mapsto(T_\alpha(f))_{\alpha\in\AA}$ extends uniquely to a bounded map
from $L_p(\NN,w)$ to $L_p(\NN,w;\ell_\infty(\AA))$, and
\begin{equation}\label{eq:matrix-maximal-extrapolation-conclusion}
 \big\|{\sup_{\alpha\in\AA}}^+T_\alpha(f)\big\|_{L_p(\NN,w)}
 \le \Phi_{p,r,d}([w]_{A_p(\NN)})\|f\|_{L_p(\NN,w)}.
\end{equation}
An explicit choice of $\Phi_{p,r,d}$ is given in \eqref{eq:matrix-Phi-prime}.
\end{theorem}

In Theorem~\ref{thm:intro-extrapolation}, letting $T_{\alpha}\equiv T$, the resulting maximal norm is
exactly $\|T(f)\|_{L_p(\NN,w)}$. Then we have the following extrapolation theorem for a single linear operator.
\begin{corollary}\label{cor:single-operator-extrapolation}
Let $T$ be a linear map from $\SS_\NN^c$ to the strongly measurable
$M_d(\C)$-valued functions. Suppose, for some $1<r<\infty$ and an increasing function
$\Phi$, one has
\[
 \|T(f)\|_{L_r(\NN,w)}
 \le\Phi([w]_{A_r(\NN)})\|f\|_{L_r(\NN,w)}
\]
for every $w\in A_r(\NN)$ and $f\in \SS_\NN^c$. Then, for every
$1<p<\infty$, $w\in A_p(\NN)$, the map
$T$ extends uniquely to a bounded map
from $L_p(\NN,w)$ to $L_p(\NN,w)$, and
\[
 \|T(f)\|_{L_p(\NN,w)}
 \le\Phi_{p,r,d}([w]_{A_p(\NN)})\|f\|_{L_p(\NN,w)},
\]
where the increasing function $\Phi_{p,r,d}$ is the same as the one in Theorem \ref{thm:intro-extrapolation}.
\end{corollary}

The key ingredient in the proof of Theorem \ref{thm:intro-extrapolation} is a noncommutative analogue of Rubio de Francia's abstract extrapolation theorem for trace-finite von Neumann algebras. More precisely, let $(E_\beta)_\beta$ be positive linear maps on a trace-finite von Neumann algebra $(\mathcal M,\tau)$; one then defines the weight class ${W_p((E_\beta)_\beta)}$ through a maximal inequality in terms of $(E_\beta)_\beta$ acting on $L_p(\mathcal M,w)$; see Definition \ref{def:weight-classes}. Under the quantitative duality Assumption \ref{ass:symmetric}, we obtain first a factorization theorem for the weights---Theorem
\ref{thm:factorization-both-sides}, and then an extrapolation result for a family of linear maps---Theorem \ref{thm:extra-abstract}. This is exactly a noncommutative analogue of Rubio de Francia's celebrated extrapolation theory \cite{Rubio1984}. 
However, because of the noncommuting products and the absence of a pointwise maximal function, the proof requires new techniques in the noncommutative setting. For instance, we identify the dense subsets of the weighted space $L_p(\mathcal M,w)$ in Lemma \ref{lem:density of M}, reformulate the condition $|\mathbf w|^{2p}\in {W_p((E_\beta)_\beta)}$ by a maximal inequality involving $(\mathbf w,\mathbf w^*)$ in Lemma \ref{lem:tech}, and establish the H\"older type testing estimates for noncommutative maximal norms in Lemmas \ref{lem:maximal-testing} and \ref{lem:compare-sandwich} etc.

With slight modifications of the arguments, a similar strategy in Section \ref{sec:abstract Rubio} also works for matrix-weighted theory associated to the family $(A_I)_{I\in\mathcal I}$. Indeed, by Theorem  \ref{thm:intro-maximal}, the weight class $W_p((A_I)_I)$ has a geometric characterization, that is, it is identical with $A_p(\mathcal N)$. From that, it is trivial that the duality Assumption \ref{ass:symmetric} is satisfied with a constant depending on the dimension $d$ of the underlying matrix. We point out the necessary modification adapted to the semifinite algebra $(\mathcal N, \int\otimes tr)$, and prove Theorem \ref{thm:intro-extrapolation} in Section \ref{sec:extrapolation on N}.

As an application, we establish weighted maximal inequalities for truncations of Calder\'on-Zygmund operators. A weighted Cotlar inequality and the matrix \(A_2\) theorem give the \(L_2\) estimate. Extrapolation then yields boundedness from \(L_p(\mathcal N,w)\) into \(L_p(\mathcal N,w;\ell_\infty)\) for every \(1<p<\infty\) and \(w\in A_p(\mathcal N)\).

\subsection{Organization of the paper}

Section~\ref{sec:preliminary} recalls the semifinite scalar and vector-valued noncommutative $L_p$ tools and the weighted realizations used in the trace-finite and semicommutative settings.  Section~\ref{sec:new Ap} introduces the geometric matrix $A_p$ class and proves its quantitative maximal characterization. Section~\ref{sec:abstract Rubio} establishes the abstract factorization and extrapolation theorems for trace-finite algebras.  Section~\ref{sec:extrapolation on N} verifies the structural assumptions for interval averages, transfers the factorization to $\NN$, proves Theorem~\ref{thm:intro-extrapolation} with explicit quantitative control, and presents an application to the Calder\'on--Zygmund operators. Finally, Section \ref{sec:problem} presents several open problems concerning extensions of our extrapolation framework and the quantitative bounds in our results.

\section{Preliminaries on tracial von Neumann algebras}\label{sec:preliminary}

This section fixes our conventions for noncommutative integration and collects the scalar, vector-valued, and weighted tools used later. In particular, we present the weighted realizations and establish some properties for trace-finite von Neumann algebras and the exceptional semifinite von Neumann algebra $\NN=L_\infty(\R)\bar\otimes M_d(\C)$ that will play an important role in obtaining noncommutative analogues of Rubio de Francia's factorization and extrapolation theory.

\subsection{Vector-valued noncommutative \texorpdfstring{$L_p$}{Lp} spaces}

We recall the facts about noncommutative integration that will be used
below; see \cite{Segal1953,Dixmier1953,Terp1981,PisierXu2003}.  Let $(\mathcal M,\tau)$ be a von Neumann algebra equipped with a normal faithful semifinite trace $\tau$. Let 
$L_0(\M)$ denote the $*$-algebra of $\tau$-measurable operators affiliated
with $\M$.  For $x\in L_0(\M)$, write
\(
   x=u|x|
\)
for its polar decomposition, and let $s(x)=s(|x|)$ denote the support
projection of $|x|$.

For $0<p<\infty$, the noncommutative $L_p$ space associated with
$(\M,\tau)$ is
\[
   L_p(\M)
   =
   \{x\in L_0(\M):\tau(|x|^p)<\infty\},
   \qquad
   \|x\|_p=\tau(|x|^p)^{1/p}.
\]
We set $L_\infty(\M)=\M$ with its operator norm.  For $p\ge1$,
$L_p(\M)$ is a Banach space.  If $1\le p<\infty$ and $p'$ is the
conjugate exponent of $p$, then the trace pairing gives the isometric
duality
\(
   L_p(\M)^*=L_{p'}(\M).
\)
We repeatedly use the noncommutative H\"older inequality: if
$0<r,p,q\le\infty$ and
\(
   1/r=1/p+1/q,
\)
then
\begin{equation}\label{eq:prelim-holder}
   \|xy\|_r\le \|x\|_p\|y\|_q.
\end{equation}
The corresponding estimate for products of finitely many factors follows
by iteration whenever the reciprocal exponents add correctly.

We shall also use the following standard tracial property. If $a_j\in L_{p_j}(\M)$,
$1\le j\le n$, and
\(
   \sum_{j=1}^n 1/{p_j}=1,
\)
then every cyclic permutation of $a_1\cdots a_n$ belongs to
$L_1(\M)$ and has the same trace.

Given $1\leq p\leq \infty$, $a,b\in L_0(\M)$ such that $ab\in L_p(\M)$, the standard polar decomposition argument gives that $\| ab\|_p=\||a|\cdot b\|_p=\|a\cdot |b^*|\|_p$.

The positive cone of $L_p(\M)$ is denoted by $L_p^+(\M)$.  If
$1\le p<\infty$, every $x\in L_p(\M)$ admits a decomposition
\[
   x=x_1-x_2+i x_3-i x_4,
   \qquad x_j\in L_p^+(\M),
\]
with
\(
   \|x_j\|_p\le \|x\|_p,
\)
for $j=1,2,3,4$.

Complex interpolation is compatible with the semifinite noncommutative
$L_p$ scale:
if $1\le p_0,p_1\le\infty$, $0<\theta<1$, and
\(
   1/p=(1-\theta)/{p_0}+\theta/{p_1},
\)
then
\[
   \bigl(L_{p_0}(\M),L_{p_1}(\M)\bigr)_\theta=L_p(\M)
\]
isometrically.

\bigskip

We next recall the maximal and summation spaces introduced by Pisier \cite{Pisier1998} and
developed by Junge \cite{Junge2002}.  Pisier's construction is formulated for hyperfinite von Neumann algebras, while Junge's work supplies the
general operator-space formulation used in the maximal--summation duality.

Let $\AA$ be an arbitrary index set and let $1\le p\le\infty$.  A family
$x=(x_\alpha)_{\alpha\in\AA}\subset L_p(\M)$ belongs to
$L_p(\M;\ell_\infty(\AA))$ if it admits a common factorization
\[
   x_\alpha=ay_\alpha b,
   \qquad
   a,b\in L_{2p}(\M),\quad
   y_\alpha\in\M,\quad
   \sup_{\alpha\in\AA}\|y_\alpha\|_\infty<\infty.
\]
Its norm is
\begin{equation}\label{eq:prelim-Lp-linfty}
   \|x\|_{L_p(\M;\ell_\infty(\AA))}
   =
   \inf_{x_\alpha=ay_\alpha b}
   \|a\|_{2p}
   \sup_{\alpha\in\AA}\|y_\alpha\|_\infty
   \|b\|_{2p}.
\end{equation}
The infimum is taken over all common factorizations of the preceding
form.  With this definition, $L_p(\M;\ell_\infty(\AA))$ is a Banach
space.  The finite-subfamily property states that
\begin{equation}\label{eq:prelim-finite-subfamilies}
   \|(x_\alpha)_{\alpha\in\AA}\|_{L_p(\M;\ell_\infty(\AA))}
   =
   \sup_{\mathcal F\subset\AA,\ \mathcal F\ \mathrm{finite}}
   \|(x_\alpha)_{\alpha\in\mathcal F}\|
      _{L_p(\M;\ell_\infty(\mathcal F))}.
\end{equation}
Thus an arbitrary family belongs to the space exactly when its finite
subfamilies have uniformly bounded maximal norms.  See
\cite[Proposition~2.1(ii) and Remark~2.4]{JungeXu2007} for more details.

For a positive family, the same result gives the order description
\begin{equation}\label{eq:prelim-positive-majorant}
   \|(x_\alpha)_\alpha\|_{L_p(\M;\ell_\infty(\mathcal A))}
   =
   \inf\bigl\{
       \|a\|_p:\;a\in L_p^+(\M),\
       x_\alpha\le a
       \text{ for every }\alpha\in \mathcal A
   \bigr\}.
\end{equation}
For this reason, we also use the customary notation
\[
   \big\|{\sup_{\alpha\in\mathcal A}}^+x_\alpha\big\|_p
   :=
   \|(x_\alpha)_\alpha\|_{L_p(\M;\ell_\infty(\mathcal A))},
\]
where the symbol ``$\sup^+$'' does not assert
the existence of an order supremum in $\M$.

For the summation space, let $\mathcal A$ be an index set and let
$1\le p<\infty$.  For
$x=(x_\alpha)_{\alpha\in\mathcal A}$ with finite support, set
\begin{equation}\label{eq:prelim-Lp-l1}
   \|x\|_{L_p(\M;\ell_1(\mathcal A))}
   :=
   \inf\Bigg\{
   \Big\|\sum_{k,\alpha}a_{k\alpha}^*a_{k\alpha}\Big\|_p^{1/2}
   \Big\|\sum_{k,\alpha}b_{k\alpha}^*b_{k\alpha}\Big\|_p^{1/2}:
   x_\alpha=\sum_k a_{k\alpha}^*b_{k\alpha}
   \text{ for every }\alpha\in\mathcal A
   \Bigg\}.
\end{equation}
The infimum is taken over finite families
$a_{k\alpha},b_{k\alpha}\in L_{2p}(\M)$. And we define the space $L_p(\M;\ell_1(\mathcal A))$ as the completion of the space containing all finitely supported families with the norm $\| \cdot\|_{L_p(\M;\ell_1(\mathcal A))}$.

We record explicitly the maximal--summation duality used below.

\begin{lemma}[{\cite[Proposition~2.1(iii)]{JungeXu2007}}]\label{lem:Lp-l1-dual}
    Let $1\leq p<\infty$ and $p'$ be the conjugate exponent of $p$. Then
    \begin{equation}
        L_p(\M;\ell_1(\mathcal A))^*=L_{p'}(\M;\ell_\infty(\mathcal A))
    \end{equation}
    isometrically via the following duality bracket
    \begin{equation}
        \big\langle (x_\alpha),(y_\alpha)\big\rangle=\sum_{\alpha\in \mathcal A}\tau(x_\alpha y_\alpha).
    \end{equation}
\end{lemma}

Throughout the paper, we will omit the index set $\mathcal A$ if no confusion can occur.

We shall use the following standard facts without reproving their operator-space constructions.

\begin{lemma}[{\cite[Chapter~3]{Pisier1998} and \cite{JungeXu2007}}]
\label{lem:Lp-l1-interpolation}
Let $\AA$ be an arbitrary index set, let
$1\le p_0,p_1\leq \infty$, and let $0<\theta<1$.  If
\[
   \frac1p=\frac{1-\theta}{p_0}+\frac{\theta}{p_1},
\]
then
\[
   \bigl(
      L_{p_0}(\M;\ell_\infty(\AA)),
      L_{p_1}(\M;\ell_\infty(\AA))
   \bigr)_\theta
   =
   L_p(\M;\ell_\infty(\AA))
\]
isometrically.
\end{lemma}

\begin{lemma}[{\cite[Section~2]{JungeXu2007}}]
\label{lem:Lp-lp-positive}
Let $\AA$ be an arbitrary index set and let $1\le p<\infty$.  If
$x=(x_\alpha)_{\alpha\in\AA}
\in L_p^+(\M;\ell_1(\AA))$, then
\[
   \|x\|_{L_p(\M;\ell_1(\AA))}
   =
   \Big\|\sum_{\alpha\in\AA}x_\alpha\Big\|_p.
\]
\end{lemma}

The following result will allow us to
test general maximal families against positive summation sequences, which is well-known to experts; see e.g. \cite{JungeXu2007}. We include here a proof with a better resulting constant.

\begin{lemma}\label{lem:l1-four-positive}
Let $1\le p<\infty$ and let $x=(x_\alpha)_\alpha$ be an element of $L_p(\M;\ell_1)$.  For every $\varepsilon>0$, there exist positive families $x^{(1)},\ldots,x^{(4)}$ and coefficients
\(
   \lambda_j=i^j
\)
, $j=1,2,3,4$ such that
\(
   x=\sum_{j=1}^4\lambda_jx^{(j)}
\)
and
\[
   \|x^{(j)}\|_{L_p(\M;\ell_1)}
   \le
   \|x\|_{L_p(\M;\ell_1)}+\varepsilon,\qquad j=1,2,3,4.
\]
If $x$ is self-adjoint coordinatewise, two positive families suffice.
\end{lemma}

\begin{proof}
Firstly, suppose that $x=(x_\alpha)_\alpha$ is self-adjoint coordinatewise. Given $\varepsilon>0$, we consider a nearly optimal factorization
\[
    x_\alpha=\sum_{k}a_{k\alpha}^*b_{k\alpha}=\sum_{k}b_{k\alpha}^*a_{k\alpha},\qquad \Big\|\sum_{k,\alpha}a_{k\alpha}^*a_{k\alpha}\Big\|_p^{1/2}
   \Big\|\sum_{k,\alpha}b_{k\alpha}^*b_{k\alpha}\Big\|_p^{1/2}\leq \|(x_\alpha)_\alpha\|_{L_p(\M;\ell_1)}+\varepsilon.
\]
We then rescale the factorization so that
\[
    \Big\|\sum_{k,\alpha}a_{k\alpha}^*a_{k\alpha}\Big\|_p=
   \Big\|\sum_{k,\alpha}b_{k\alpha}^*b_{k\alpha}\Big\|_p=\Big\|\sum_{k,\alpha}a_{k\alpha}^*a_{k\alpha}\Big\|_p^{1/2}
   \Big\|\sum_{k,\alpha}b_{k\alpha}^*b_{k\alpha}\Big\|_p^{1/2}.
\]
Define $u=(u_\alpha)_\alpha$ and $v=(v_\alpha)_\alpha$ by
\[
    u_\alpha =\frac{1}{4}\sum_k (a_{k\alpha}+b_{k\alpha})^*(a_{k\alpha}+b_{k\alpha}),\qquad v_\alpha =\frac{1}{4}\sum_k (a_{k\alpha}-b_{k\alpha})^*(a_{k\alpha}-b_{k\alpha}),
\]
so both $u$ and $v$ are positive. Moreover, we have
\[
    u_\alpha-v_\alpha=x_\alpha ,\qquad u_\alpha+v_\alpha=\frac{1}{2}\sum_k (a_{k\alpha}^*a_{k\alpha}+b_{k\alpha}^*b_{k\alpha}).
\]
Then by Lemma~\ref{lem:Lp-lp-positive}
\[
    \begin{aligned}
        \| u\|_{L_p(\M;\ell_1)}&\leq \|u+v\|_{L_p(\M;\ell_1)} \leq \frac{1}{2}\Big\| \sum_{k,\alpha} (a_{k\alpha}^*a_{k\alpha}+b_{k\alpha}^*b_{k\alpha}) \Big\|_p \\ 
        &\leq \frac{1}{2}\Big[\Big\| \sum_{k,\alpha} a_{k\alpha}^*a_{k\alpha} \Big\|_p+\Big\| \sum_{k,\alpha} b_{k\alpha}^*b_{k\alpha} \Big\|_p\Big]\leq \|(x_\alpha)_\alpha\|_{L_p(\M;\ell_1)}+\varepsilon.
    \end{aligned}
\]

Given a general $x\in L_p(\M;\ell_1)$, applying this to $\Re x$ and $\Im x$, both
contractive on $L_p(\M;\ell_1)$, gives the general assertion. 
\end{proof}

We shall also use the following well-known balanced module H\"older inequality. See, for instance,  \cite{Pisier1998,JungeXu2007}.

\begin{lemma}
\label{lem:vector-holder}
Let $1\le r,p,s\le\infty$ and $1/r=1/p+1/s$. For 
$a,b\in L_{2s}(\M)$ and $(x_\alpha)_\alpha\in L_p(\M;\ell_\infty)$,
\begin{equation}\label{eq:vector-holder}
 \|(ax_\alpha b)_\alpha\|_{L_r(\M;\ell_\infty)}
 \le \|a\|_{2s}\|(x_\alpha)_\alpha\|_{L_p(\M;\ell_\infty)}\|b\|_{2s}.
\end{equation}
In particular, if $u\in L_s^+(\M)$, then
\[
 \big\|(u^{1/2}x_\alpha u^{1/2})_\alpha\big\|_{L_r(\M;\ell_\infty)}
 \le \|u\|_s\|(x_\alpha)_\alpha\|_{L_p(\M;\ell_\infty)}.
\]
\end{lemma}

\begin{proof}
Take a common factorization
\[
   x_\alpha=c y_\alpha d,
   \qquad c,d\in L_{2p}(\M),
   \qquad \sup_\alpha\|y_\alpha\|_\infty<\infty.
\]
Then $ax_\alpha b=(ac)y_\alpha (db)$ forms a factorization of $(ax_\alpha b)_{\alpha}$. Since $1/(2r)=1/(2s)+1/(2p)$, H\"older's inequality gives
\[
    \begin{aligned}
        \|(ax_\alpha b)_\alpha\|_{L_r(\M;\ell_\infty)}
 &\le\inf_{ax_\alpha b=(ac)y_\alpha (db)} \| ac\|_{2r}\big(\sup_{\alpha}\| y_\alpha\|_\infty\big)\| db\|_{2r} \\
 &\leq \inf_{x_\alpha=c y_\alpha d} \|a\|_{2s}\cdot \| c\|_{2p}\big(\sup_{\alpha}\| y_\alpha\|_\infty\big)\| d\|_{2p}\cdot \| b\|_{2s} \\
 &= \|a\|_{2s}\|(x_\alpha)_\alpha\|_{L_p(\M;\ell_\infty)} \| b\|_{2s},
    \end{aligned}
\]
which gives the estimate as desired.
\end{proof}

\subsection{{Weighted \texorpdfstring{$L_p$}{Lp}-spaces and more properties on trace-finite von Neumann algebras}} 
In this subsection, the noncommutative measure space $(\mathcal M,\tau)$ is assumed to be finite.
We say $w\in L_0^+(\M)$ is an operator weight if it has full support. Given an operator weight $w$, it is not hard to verify that $w^{-1}$ is also an operator weight since $(\M,\tau)$ is finite. For $1\le p<\infty$, and $x\in L_0(\M)$, its weighted $L_p$-norm is defined as
\begin{equation}\label{eq:prelim-weighted-Lp}
   \|x\|_{L_p(\M,w)}=
   \|w^{1/2p}xw^{1/2p}\|_p,
\end{equation}
and the weighted $L_p$-space is defined to be the space of measurable operators with finite weighted $L_p$-norm. That is,
\[
    L_p(\M,w):=\{ x\in L_0(\M):\|x\|_{L_p(\M,w)}<\infty\}.
\]

The map
\[
   \begin{aligned}
       J_{p,w}:L_p(\M,w)&\longrightarrow L_p(\M) \\
       x&\longmapsto w^{1/2p}xw^{1/2p}
   \qquad
   \end{aligned}
\]
is an onto isometry.  Indeed, it is an isometry by definition.  If
$y\in L_p(\M)$, then
\(
   x=w^{-1/2p}yw^{-1/2p}
\)
belongs to $L_0(\M)$, and 
\(
   J_{p,w}(x)=y.
\)
It follows from the completeness of $L_p(\M)$ that $L_p(\M,w)$ is a
Banach space. Put
 \[
    e_n(w)=1_{[1/n,n]}(w),
    \qquad
    \mathscr C(w)=\bigcup_{n\ge1}e_n(w)\M e_n(w).
 \]

 \begin{lemma}\label{lem:density of M}
     For every $1\leq p<\infty$, $\mathscr C(w)$ is dense in $L_p(\M,w)$. As a consequence, $\M\cap L_p(\M,w)\subset L_p(\M,w)$ is dense.
 \end{lemma}

 \begin{proof}
     Given $a\in L_p(\M,w)$, we only need to show that
     \begin{equation}\label{equ:inf represent of dense}
         \inf_{z\in \mathscr C(w)}\| J_{p,w}(z-a)\|_p=0.
     \end{equation}
      Firstly, we show that $J_{p,w}(\mathscr C(w))=\mathscr C(w)$. On the one hand, given $x\in \mathscr C(w)$, there exists some $m\in\N_+$ and $y\in\M$ such that $x= e_m(w)ye_m(w)$. Then
     \[
         J_{p,w}(x)=e_m(w)w^{1/2p}yw^{1/2p}e_m(w)=e_m(w)\big( e_m(w)w^{1/2p}yw^{1/2p}e_m(w)\big)e_m(w)\in \mathscr C(w).
     \]
     On the other hand, we know that $e_m(w)w^{-1/2p}yw^{-1/2p}e_m(w)\in \mathscr C(w)$ and
     \[
         J_{p,w}(e_m(w)w^{-1/2p}yw^{-1/2p}e_m(w))=x,
     \]
     which proves the claim. Then we have $\mathscr C(w)\subset L_p(\M,w)$.

     Thus the equation \eqref{equ:inf represent of dense} is equivalent to
     \[
         \inf_{z\in \mathscr C(w)}\| z-w^{1/2p}aw^{1/2p}\|_p=0,
     \]
     which is equivalent to the density of $\mathscr C (w)$ in the unweighted space $L_p(\M)$. The definition of an operator weight gives $e_n(w)\uparrow1$. Then for every
 $x\in L_p(\M)$,
 \begin{equation}\label{eq:weighted-corner-convergence}
    e_n(w)xe_n(w)\longrightarrow x
    \qquad\text{in }L_p(\M),
 \end{equation}
which is the standard projection approximation in the noncommutative $L_p$-space. Moreover, $\M\subset L_p(\M)$ is dense; thus
 \[
     \inf_{z\in \M}\| e_n(w)(z-x)e_n(w)\|_p\leq \inf_{z\in \M}\| z-x\|_p=0,
 \]
 which gives the density as desired.

 Finally, $\mathscr C(w)\subset \M$ gives the density of $\M\cap L_p(\M,w)$ in $L_p(\M,w)$.
 \end{proof}

Now we are able to define the vector-valued weighted noncommutative $L_p$-spaces.
Given $1\leq p\leq\infty$ and an arbitrary index set $\AA$, the space $L_p(\M,w;\ell_\infty(\AA))$ is defined to be the set of families $(x_\alpha)_{\alpha\in\mathcal A}\subset L_0(\M)$ with finite norm
\begin{equation}\label{eq:prelim-weighted-vector-space}
 \|(x_\alpha)_{\alpha\in\AA}\|_{L_p(\M,w;\ell_\infty(\AA))}
 :=\|(w^{1/2p}x_\alpha w^{1/2p})_{\alpha\in\AA}\|_{L_p(\M;\ell_\infty(\AA))}<\infty.
\end{equation}

\bigskip

Next, we record the scalar and maximal H\"older type testing estimates that will be used in later sections.

\begin{lemma}\label{lem:maximal-testing}
Let $1<p<\infty$, $(x_\alpha)_\alpha\in L_p(\M;\ell_\infty)$ and $x\in L_p(\M)$.
\begin{enumerate}[label=\textnormal{(\roman*)}]
\item If $1<p<r<\infty$, and $1/p=1/r+1/s$, then, for every $\varepsilon>0$, there is a full-support operator $u\in L_s^+(\M)$ such that
\begin{equation}\label{eq:single-inverse-testing}
 \|u^{-1/2}xu^{-1/2}\|_r\|u\|_s \le2\|x\|_p+\varepsilon.
\end{equation}
Moreover, if $x$ is positive, then the constant $2$ can be replaced by $1$.

\item If $1<r<p<\infty$ and $1/r=1/p+1/s$, then
\begin{equation}\label{eq:general-maximal-testing}
 \sup_{\substack{ u\in L_s^+(\M)\\s(u)=1}}
 \frac{
  \big\|{\sup_\alpha}^+u^{1/2}x_\alpha u^{1/2}\big\|_r
 }{\|u\|_s}
 \le\big\|{\sup_\alpha}^+x_\alpha\big\|_p\le
 4\sup_{\substack{ u\in L_s^+(\M)\\s(u)=1}}
 \frac{
  \big\|{\sup_\alpha}^+u^{1/2}x_\alpha u^{1/2}\big\|_r
 }{\|u\|_s}.
\end{equation}

\end{enumerate}
\end{lemma}

\begin{proof}
For (i), let $x=q|x|$ be the polar decomposition of $x$ and set
\[
 u_\delta=|x|^{p/s}+q|x|^{p/s}q^*+\delta1.
\]
Trace-finiteness makes $u_\delta$ invertible in $L_s^+(\M)$.  Since
$|x|^{p/s}\le u_\delta$ and $q|x|^{p/s}q^*\le u_\delta$, Douglas factorization makes the
two outside factors in
\[
 u_\delta^{-1/2}xu_\delta^{-1/2}
 =\bigl(u_\delta^{-1/2}q|x|^{p/2s}\bigr)|x|^{p/r}
   \bigl(|x|^{p/2s}u_\delta^{-1/2}\bigr)
\]
contractions.  Therefore
\[
 \|u_\delta^{-1/2}xu_\delta^{-1/2}\|_r
 \le\|x\|_p^{p/r},
 \qquad
 \limsup_{\delta\downarrow0}\|u_\delta\|_s
 \le2\|x\|_p^{p/s}.
\]
The identity $p/r+p/s=1$ proves the estimate. If $x$ is positive, set $u_\delta=x^{p/s}+\delta 1$ and follow the proof above to obtain (i).

The first inequality in (ii) follows immediately from Lemma~\ref{lem:vector-holder}. To prove the reverse inequality,
assume that the supremum on its right is finite. Lemma~\ref{lem:Lp-l1-dual} gives that
\[
    \big\|{\sup_\alpha}^+x_\alpha\big\|_p=\sup_{(y_\alpha)_\alpha:\| (y_\alpha)\|_{L_{p'}(\M;\ell_1)}=1}\Big| \sum_{\alpha}\tau(y_\alpha x_\alpha)\Big|.
\]
Using the decomposition in Lemma~\ref{lem:l1-four-positive}, we obtain
 \begin{equation}\label{equ:one to four}
     \begin{aligned}
         \big\|{\sup_\alpha}^+x_\alpha\big\|_p&\leq\sup_{(y_\alpha)_\alpha:\| (y_\alpha)\|_{L_{p'}(\M;\ell_1)}=1}\sum_{j=1}^4\Big| \sum_{\alpha}\tau(y_\alpha^{(j)} x_\alpha)\Big| \\
         &\leq 4\sup_{\substack{(y_\alpha)_\alpha \in L_{p'}^+(\M;\ell_1)\\ \| (y_\alpha)\|_{L_{p'}(\M;\ell_1)}=1} }\Big| \sum_{\alpha}\tau(y_\alpha x_\alpha)\Big|.
     \end{aligned}
 \end{equation}
Given $(y_\alpha)_\alpha\in L_{p'}^+(\M;\ell_1)$ with the norm $\|(y_\alpha)_\alpha\|_{L_{p'}(\M;\ell_1)}=1$, put $\widetilde{y}=\sum_\alpha y_\alpha$. Then by (i), for any $\varepsilon>0$, we can find some full-support $u\in L_s^+(\M)$ such that
\[
    \| u^{-1/2}\widetilde{y}u^{-1/2}\|_{r'}\| u\|_s\leq \| \widetilde{y}\|_{p'}+\varepsilon=1+\varepsilon.
\]
H\"older's inequality and maximal--summation duality give
\[
    \begin{aligned}
        \Big|\sum_\alpha\tau(y_\alpha x_\alpha)\Big| &=\Big|\sum_\alpha\tau(u^{-1/2}y_\alpha u^{-1/2}u^{1/2}x_\alpha u^{1/2})\Big|\\
        &\leq \|(u^{-1/2}y_\alpha u^{-1/2})_\alpha\|_{L_{r'}(\M;\ell_1)}\cdot \big\| {\sup_\alpha}^+ u^{1/2}x_\alpha u^{1/2}\big\|_r \\
        &\leq(1+\varepsilon )\cdot \frac{\big\| {\sup_\alpha}^+ u^{1/2}x_\alpha u^{1/2}\big\|_r}{\|u\|_s}\\
        &\leq (1+\varepsilon )\cdot\sup_{\substack{ u\in L_s^+(\M)\\s(u)=1}} \frac{\big\| {\sup_\alpha}^+ u^{1/2}x_\alpha u^{1/2}\big\|_r}{\|u\|_s},
    \end{aligned}
\]
which holds for any $\varepsilon>0$. Thus we have
\begin{equation}\label{equ:dual singal}
    \Big|\sum_\alpha\tau(y_\alpha x_\alpha)\Big|\leq \sup_{\substack{ u\in L_s^+(\M)\\s(u)=1}}
 \frac{
  \big\|{\sup_\alpha}^+u^{1/2}x_\alpha u^{1/2}\big\|_r
 }{\|u\|_s}.
\end{equation}
Combining \eqref{equ:one to four} and \eqref{equ:dual singal}, we get the conclusion.  
\end{proof}

The following lemma establishes a monotonicity property of noncommutative maximal norms.
\begin{lemma}\label{lem:compare-sandwich}
Let $1\le p<\infty$, $0<u\le v$ be positive invertible elements of $L_0(\M)$, and let $(x_\alpha)_{\alpha\in\AA}\subset L_0(\M)$.  If
\(
   (v^{1/2}x_\alpha v^{1/2})_{\alpha\in\AA}
   \in L_p(\M;\ell_\infty(\AA)),
\)
then
\[
\begin{aligned}
   \big\|
      {\sup_\alpha}^+u^{1/2}x_\alpha u^{1/2}
   \big\|_{p}
   \le
   \big\|
     {\sup_\alpha}^+ v^{1/2}x_\alpha v^{1/2}
   \big\|_{p}
\end{aligned}
\]
Moreover,
\(
   v^{-1}\le u^{-1},
\)
and, if
\(
   (u^{-1/2}x_\alpha u^{-1/2})_{\alpha\in\AA}
   \in L_p(\M;\ell_\infty(\AA)),
\)
then
\[
\begin{aligned}
   \big\| {\sup_\alpha}^+
      v^{-1/2}x_\alpha v^{-1/2}
   \big\|_{p}
   \le
   \big\|{\sup_\alpha}^+
      u^{-1/2}x_\alpha u^{-1/2}
   \big\|_{p}.
\end{aligned}
\]
\end{lemma}

\begin{proof}
The operator order $u\le v$ gives a contraction $c\in\M$ such that
$u^{1/2}=cv^{1/2}$.  Hence, in $L_0(\M)$,
\(
 u^{1/2}x_\alpha u^{1/2}
 =c\bigl(v^{1/2}x_\alpha v^{1/2}\bigr)c^*
\). 
Lemma~\ref{lem:vector-holder}, with $s=\infty$, $a=c$, and $b=c^*$,
proves the first norm inequality.

Full support of $v$ and $u$ gives the reverse order $0< v^{-1}\leq u^{-1}$. Applying the first part to the pair $(v^{-1}, u^{-1})$ proves the reverse estimate.
\end{proof}

\subsection{Weighted \texorpdfstring{$L_p$}{Lp}-spaces and more properties on \texorpdfstring{$\mathcal N$}{calN}} Even though we do not know how to properly define the weighted $L_p$-spaces and establish the related properties on general noncommutative measure spaces, it is possible to do that on the exceptional semifinite spaces,
\[
 \NN=L_\infty(\R)\bar\otimes M_d(\C),
\]
equipped with the canonical tensor trace. In this case, it is well-known that 
\begin{equation}\label{equ:canonic isomorphism}
 L_p(\NN)=L_p(\R;S_p^d)
\end{equation}
isometrically for $1\le p<\infty$.

A \emph{matrix weight} is, throughout this paper, a strongly measurable
function $w:\R\to M_d(\C)$ which is positive definite almost everywhere
and locally integrable.  All powers are formed by pointwise functional
calculus. Given $1\le p<\infty$ and a strongly measurable matrix-valued
function $f$, we define the weighted $L_p$-norm
\begin{equation}\label{eq:matrix-global-weighted-norm}
\begin{aligned}
 \|f\|_{L_p(\NN,w)}
 &=\left(\int_\R
   \|w(t)^{1/2p}f(t)w(t)^{1/2p}\|_{S_p^d}^p\,dt\right)^{1/p},
\end{aligned}
\end{equation}
and the related weighted $L_p$-space $L_p(\NN,w)$ consists of all such functions with finite weighted $L_p$-norm. Similar to the abstract case, the map
\[
    \begin{aligned}
        J_{p,w}:L_p(\NN,w)&\longrightarrow L_p(\NN) \\
        f&\longmapsto w^{1/2p}fw^{1/2p}
    \end{aligned}
\]
is an onto isometry. Thus $L_p(\NN,w)$ is a Banach space. Notice that for a matrix weight $w$, it may not be a measurable operator, i.e., $w\notin L_0(\NN)$. Similarly, $L_p(\NN,w)\not\subset L_0(\NN)$. The multiplication $w^{1/2p}fw^{1/2p}$ should be understood pointwise, and for $f\in L_p(\NN,w)$, we may identify $w^{1/2p}fw^{1/2p}$ as an element of $L_p(\NN)\subset L_0(\NN)$ since it has finite $L_p(\NN)$-norm. However, a matrix weight $w$ is locally integrable so that we have the following suitable dense class.

Define $\SS_{\NN}^c$ as the subspace of $\NN$ with  elements having compact support, equivalently,
\[
    \SS_\NN^c=\bigcup_{m\geq 1}L_\infty(-m,m)\bar\otimes M_d(\C).
\]
One can verify that $\SS_\NN^c\subset L_p(\NN,w)$ since $w$ is locally integrable. The following lemma shows that it is a dense subspace of $L_p(\NN,w)$.

\begin{lemma}\label{lem:density of N}
For $1\le p<\infty$, $\SS_\NN^c$ is dense in
$L_p(\NN,w)$.
\end{lemma}

\begin{proof}
Given $f\in L_p(\NN,w)$, put
\[
 E_{m,k}=(-m,m)\cap
 \{t:\|f(t)\|_{S_\infty^d}\le k\},
 \qquad f_{m,k}=\chi_{E_{m,k}}f.
\]
Then $f_{m,k}\in \SS_\NN^c$, and
\[
 \|f-f_{m,k}\|_{L_p(\NN,w)}=\|  J_{p,w}(f-f_{m,k})\|_p
 =\|\chi_{\R\setminus E_{m,k}}J_{p,w}f\|_p\longrightarrow 0
\]
when $m,k\to +\infty$ since the sets $E_{m,k}$ increase to $\R$ up to a null set.
\end{proof}

\bigskip

For $1\le p< \infty$ and an arbitrary index set
$\AA$, define $L_p(\mathcal N,w;\ell_\infty(\AA))$ as the set of families of strongly measurable functions $(f_\alpha)_{\alpha\in\mathcal A}$ with finite norm 
\begin{equation}\label{eq:matrix-global-vector-weighted-norm}
 \|(f_\alpha)_{\alpha\in\AA}\|_{L_p(\NN,w;\ell_\infty(\AA))}
 :=\|(w^{1/2p}f_\alpha w^{1/2p})_{\alpha\in\AA}\|_{L_p(\NN;\ell_\infty(\AA))}.
\end{equation}

\bigskip

A strongly measurable matrix-valued function $w$ is called \emph{invertible} (resp. {\it positive}) when
$w(t)$ is invertible (resp. positive) for almost every $t\in\mathbb R$. The H\"older type testing property and monotonicity property have the following direct
semicommutative counterpart.

\begin{lemma}\label{lem:N-testing}
Let $1<p<\infty$, $(f_\alpha)_\alpha\in L_p(\NN;\ell_\infty)$ and $f\in L_p(\NN)$.
\begin{enumerate}[label=\textnormal{(\roman*)}]
\item If $1<p<r<\infty$, and $1/p=1/r+1/s$, then, for every $\varepsilon>0$, there is an invertible function $u\in L_s^+(\NN)$ such that
\begin{equation}\label{eq:single-inverse-testing-NN}
 \|u^{-1/2}fu^{-1/2}\|_r\|u\|_s \le2\|f\|_p+\varepsilon,
\end{equation}
where the inverse and multiplication should be understood pointwise. Moreover, if $f$ is positive, then the constant $2$ can be replaced by $1$.

\item If $1<r<p<\infty$ and $1/r=1/p+1/s$, then
\begin{equation}\label{eq:general-maximal-testing-NN}
 \sup_{\substack{ u\in L_s^+(\NN)\\s(u)=1}}
 \frac{
  \big\|{\sup_\alpha}^+u^{1/2}f_\alpha u^{1/2}\big\|_r
 }{\|u\|_s}
 \le\big\|{\sup_\alpha}^+f_\alpha\big\|_p\le
 4\sup_{\substack{ u\in L_s^+(\NN)\\s(u)=1}}
 \frac{
  \big\|{\sup_\alpha}^+u^{1/2}f_\alpha u^{1/2}\big\|_r
 }{\|u\|_s}.
\end{equation}

\end{enumerate}
\end{lemma}
\begin{proof}
    The only change from the proof of Lemma~\ref{lem:maximal-testing} is that in (i) a small multiple of $1$
	can no longer be used to make a testing element invertible.  On $\NN$,
	replace that term by
\(
 \delta e^{-|t|}1_d.
\)
\end{proof}

\begin{lemma}\label{N-test-2}
    Let $1\le p<\infty$, $0<u\le v$ be positive invertible strongly measurable functions, and let $(f_\alpha)_{\alpha\in\AA}$ be a family of strongly measurable functions.  If
\(
   (v^{1/2}f_\alpha v^{1/2})_{\alpha\in\AA}
   \in L_p(\NN;\ell_\infty(\AA)),
\)
then
\[
\begin{aligned}
   \big\|
      {\sup_\alpha}^+u^{1/2}f_\alpha u^{1/2}
   \big\|_{p}
   \le
   \big\|
     {\sup_\alpha}^+ v^{1/2}f_\alpha v^{1/2}
   \big\|_{p}
\end{aligned}
\]
Moreover,
\(
   v^{-1}\le u^{-1},
\)
and, if
\(
   (u^{-1/2}f_\alpha u^{-1/2})_{\alpha\in\AA}
   \in L_p(\NN;\ell_\infty(\AA)),
\)
then
\[
\begin{aligned}
   \big\| {\sup_\alpha}^+
      v^{-1/2}f_\alpha v^{-1/2}
   \big\|_{p}
   \le
   \big\|{\sup_\alpha}^+
      u^{-1/2}f_\alpha u^{-1/2}
   \big\|_{p}.
\end{aligned}
\]
\end{lemma}

\begin{proof}
Compared with Lemma~\ref{lem:compare-sandwich}, the only difference is that $0<u(t)\le v(t)$ implies that $\varphi(t)=u(t)^{1/2}v(t)^{-1/2}$ is a contraction and $u(t)^{1/2}=\varphi(t)v(t)^{1/2}$, where $\varphi$ is strongly measurable.
\end{proof}

To summarize, the unweighted scalar and vector-valued results of this section hold for arbitrary semifinite traces. The affiliated-operator weighted arguments in Subsection~2.2 are presently restricted to trace-finite von Neumann algebras only. But  the pointwise weighted realization and Lemma~\ref{lem:N-testing} apply directly to the concrete algebra $\NN$.

\section{A new matrix weight class and its maximal characterization}\label{sec:new Ap}

The aim of this section is to present a new type of matrix weights, and their characterizations via  noncommutative
weighted  Hardy--Littlewood maximal inequalities. For this purpose, we will relate our new $A_p$ class with the one introduced by Volberg \cite{Volberg1997}.

\subsection{Volberg's matrix weights}
Let us recall the definition of the matrix $A_p$ class given by Volberg. Given a matrix weight $w:\R\to M_{d}(\C)$, $t\in \R$, and $1<p<\infty$, we define a norm by
\[
    \rho_{w,p,t}(u):=\| w^{1/p}(t)u\|_{\ell_2^{d}}, \qquad u\in \ell_2^{d},
\]
and its dual norm by
\[
    \rho_{w,p,t}^*(u):=(\rho_{w,p,t})^*(u)=\sup_{0\neq v\in \ell_2^{d}}\frac{|\langle u,v\rangle_{\ell_2^{d}}|}{\rho_{w,p,t}(v)}.
\]
{It is easy to check $\rho_{w,p,t}^*(u)=\| w^{-1/p}(t)u\|_{\ell_2^{d}}$.}

{Let $\II$ denote the family of bounded intervals in $\R$.} Given $I\in \II$, set
\[
    \rho_{w,p,I}(u):=\Big(\frac{1}{|I|}\int_I \rho_{w,p,t}(u)^pdt\Big)^{1/p}, \qquad \rho_{w,p,I}^*(u):=\Big(\frac{1}{|I|}\int_I \rho_{w,p,t}^*(u)^{p'}dt\Big)^{1/p'}.
\]
Say $w\in A_p(\R;M_{d}(\C))$ if there exists a constant $C>0$, such that
\[
    \rho_{w,p,I}^*\leq C^{1/p}(\rho_{w,p,I})^*
\]
for all $I\in \II$. The least constant $C$ is denoted by $[w]_{A_p(\R;M_{d}(\C))}$. Readers may refer to \cite{Volberg1997} for more details. 

We will use the following quantitative version of Goldberg's maximal inequality given by Isralowitz and Moen in \cite{IsraMoen2019}:
\begin{lemma}\label{lem:quantitative goldberg}
    Let $w:\R\to M_{d}(\C)$ be a matrix weight and let the maximal operator $M_w$ be defined  in \eqref{equ:goldberg maximal}. Then we have the following estimate:
    \begin{equation}\label{equ:Goldberg quantitative}
        \| M_w\|_{L_p(\R;\ell_2^{d})\to L_p(\R)}\le C[w]_{A_p(\R;M_{d}(\C))}^{p'/p}.
    \end{equation}
    Moreover, it is sharp with respect to the exponent $p'/p$ of the $A_p$ characteristic. We denote by $\GG_{p,d}$ the least constant $C$ such that \eqref{equ:Goldberg quantitative} holds for every $w\in A_p(\R;M_d(\C))$.
\end{lemma}

\begin{remark}
    Following the proof given in \cite{IsraMoen2019}, one may check that, for instance, $\GG_{p,d}\lesssim_p d^{7/2}$. But the authors believe that this exponent $7/2$ is far from optimal.
\end{remark}

\subsection{The new matrix \texorpdfstring{$A_p$}{Ap} class on \texorpdfstring{$\NN$}{N}}
We give our new matrix $A_p$ class as follows, which is motivated by Volberg's definition above. {Given a matrix weight $w:\R\to M_{d}(\C)$, $t\in \R$, and $1<p<\infty$}, we define a norm on $M_d(\C)$ by
\begin{equation}\label{eq:matrix-pointwise-metric}
 \rho_{w,p,t}(A)=\|w(t)^{1/2p}Aw(t)^{1/2p}\|_{S_p^d},
 \qquad A\in M_d(\C),
\end{equation}
and define its dual by
\[
\rho_{w,p,t}^*(A):=(\rho_{w,p,t})^*(A):=\sup_{0\neq B\in M_d(\C)}\frac{|\tr(B^*A)|}{\rho_{w,p,t}(B)}
\]
H\"older's inequality and the positivity of $w$ give
\[
 \rho_{w,p,t}(A)^p\le \|A\|_{S_\infty^d}^p\tr(w(t)),
\]
so $\rho_{w,p,t}(A)^p$ is locally integrable.

\begin{lemma}\label{lem:matrix-pointwise-dual}
  For almost every $t\in\R$ and every $A\in M_d(\C)$,
\begin{equation}\label{eq:matrix-pointwise-dual}
 \rho_{w,p,t}^*(A)
 =\|w(t)^{-1/2p}Aw(t)^{-1/2p}\|_{S_{p'}^d}.
\end{equation}
\end{lemma}

\begin{proof}
Fix $t$ such that $w(t)>0$.  The change of variables
$C=w(t)^{1/2p}Bw(t)^{1/2p}$ is bijective on $M_d(\C)$, and finite-dimensional
trace cyclicity gives
\[
 \tr(B^*A)
 =\tr\bigl(C^*w(t)^{-1/2p}Aw(t)^{-1/2p}\bigr).
\]
Taking the supremum over $C$ and using the duality between $S_p^d$ and
$S_{p'}^d$ proves the formula.  The required assumption holds almost
everywhere by the definition of a matrix weight.
\end{proof}

For $I\in\II$, set
\begin{equation}\label{eq:matrix-averaged-metrics}
 \rho_{w,p,I}(A)
 =\left(\frac1{|I|}\int_I\rho_{w,p,t}(A)^p\,dt\right)^{1/p},\qquad 
 \rho_{w,p,I}^*(A)
 =\left(\frac1{|I|}\int_I\rho_{w,p,t}^*(A)^{p'}\,dt\right)^{1/p'}.
\end{equation}
The second quantity is allowed to be infinite.

\begin{lemma}\label{lem:matrix-averaged-duality}
For every $I\in\II$ and $A\in M_d(\C)$,
\begin{equation}\label{eq:matrix-one-sided-geometry}
 (\rho_{w,p,I})^*(A)\le \rho_{w,p,I}^*(A).
\end{equation}
\end{lemma}

\begin{proof}
Pointwise duality gives
$|\tr(B^*A)|\le\rho_{w,p,t}(B)\rho_{w,p,t}^*(A)$.  If
$\rho_{w,p,I}^*(A)=\infty$, there is nothing to prove.  Otherwise, 
H\"older's inequality gives
\[
 |\tr(B^*A)|
 \le \frac{1}{|I|}\int_I\rho_{w,p,t}(B)\rho_{w,p,t}^*(A) dt\leq \rho_{w,p,I}(B)\rho_{w,p,I}^*(A).
\]
Taking the supremum over $B\ne0$ proves
\eqref{eq:matrix-one-sided-geometry}.
\end{proof}

\begin{definition}\label{def:matrix-Ap-global-local}
We say that $w\in A_p(\NN)$ if there is a constant $C\ge1$ such that
\begin{equation}\label{eq:matrix-Ap-global}
 \rho_{w,p,I}^*(A)
 \le C^{1/p}(\rho_{w,p,I})^*(A)
\end{equation}
for all $I\in\II$ and $A\in M_d(\C)$.  The least  constant $C$ such that \eqref{eq:matrix-Ap-global} holds is denoted by $[w]_{A_p(\NN)}$.
\end{definition}

Given $w\in A_p(\NN)$ and $A\in M_d(\C)$, Definition~\ref{def:matrix-Ap-global-local} ensures that $\rho^*_{w,p,t}(A)^{p'}$ is locally integrable and hence $w^{-p'/p}$ is also a matrix weight. The geometric characteristic has the exact duality required later.

\begin{proposition}\label{prop:matrix-Ap-dual}
Given $1<p<\infty$, then $w\in A_p(\NN)$ is equivalent to $w^{-p'/p}\in A_{p'}(\NN)$. Moreover, quantitatively,
\begin{equation}\label{eq:Ap-dual-characteristic}
 [w^{-p'/p}]_{A_{p'}(\NN)}
 =[w]_{A_p(\NN)}^{p'/p}.
\end{equation}
\end{proposition}

\begin{proof}
By Lemma~\ref{lem:matrix-pointwise-dual}, the pointwise norm
associated with $w^{-p'/p}$ at exponent $p'$ is $\rho_{w,p,t}^*$, and its
dual is $\rho_{w,p,t}$. The defining inequality is
\[
 \rho_{w,p,I}^*\le [w]_{A_p(\NN)}^{1/p}(\rho_{w,p,I})^*.
\]
Dualizing this inequality between finite-dimensional normed spaces gives
\[
 \rho_{w,p,I}\le [w]_{A_p(\NN)}^{1/p}(\rho_{w,p,I}^*)^*,
\]
which is the $A_{p'}$ inequality for $w^{-p'/p}$ with the same root
constant. Repeating the argument with $w^{-p'/p}$ gives the converse and shows $[w]_{A_p(\NN)}^{1/p}=[w^{-p'/p}]_{A_{p'}(\NN)}^{1/p'}$, which is
\eqref{eq:Ap-dual-characteristic}.
\end{proof}

In this section, the key idea is that one may relate $w\in A_p(\mathcal N)$ with Volberg's matrix weight $\mathbb W\in A_p(\mathbb R;M_{d^2}(\mathbb C))$ in the following way
\begin{equation}\label{equ:mathbb W}
    \mathbb W(t): A\longmapsto w^{1/2}(t)Aw^{1/2}(t),\qquad A\in S^d_2(\C).
\end{equation}
Then by functional calculus, we have
\begin{equation}\label{eq:operator-weight-powers}
    \mathbb W^{\pm 1/p}(t)(A)= w^{\pm1/2p}(t)Aw^{\pm1/2p}(t),\qquad A\in S_2^d(\C).
\end{equation}

Based on this observation, we can compare our $A_p(\NN)$ class with Volberg's $A_p(\R;M_{d^2}(\C))$ class as follows.
\begin{lemma}\label{lem:compare N and ell_2^d2}
    Given a matrix weight $w:\R\to M_d(\C)$, consider the matrix weight $\mathbb W$ defined by \eqref{equ:mathbb W}. Then we have
    \begin{equation}
        d^{-|1-\frac{p}{2}|}[\mathbb W]_{A_p(\R;M_{d^2}(\C))}\leq [w]_{A_p(\NN)}\leq d^{|1-\frac{p}{2}|}[\mathbb W]_{A_p(\R;M_{d^2}(\C))}.
    \end{equation}
\end{lemma}

\begin{proof}
Define 
\[
        \kappa_p(d)=d^{\max\{0,\frac{1}{p}-\frac{1}{2}\}},\qquad \eta_p(d)=d^{\min\{0,\frac{1}{p}-\frac{1}{2}\}}.
    \]
    On the one hand, we have
    \[
        \rho_{w,p,t}(A)=\|w(t)^{1/2p}Aw(t)^{1/2p}\|_{S_p^d}\leq \kappa_p(d)\| \mathbb W^{1/p}(t)A\|_{\ell_2^{d^2}}=\kappa_p(d)\rho_{\mathbb W,p,t}(A),
    \]
    \[
        \rho_{w,p,t}(A)\geq \eta_p(d)\| \mathbb W^{1/p}(t)A\|_{\ell_2^{d^2}}=\eta_p(d)\rho_{\mathbb W,p,t}(A),
    \]
    thus for any $I\in\II$,
    \[
        \eta_p(d)\rho_{\mathbb W,p,I}(A)\leq \rho_{w,p,I}(A)=\Big(\frac{1}{|I|}\int_I \rho_{w,p,t}(A)^pdt \Big)^{1/p}\leq \kappa_p(d)\rho_{\mathbb W,p,I}(A).
    \]
    Taking dual norms gives
    \[
        \kappa_p^{-1}(d)(\rho_{\mathbb W,p,I})^*(A)\leq (\rho_{w,p,I})^*(A)\leq \eta_p^{-1}(d)(\rho_{\mathbb W,p,I})^{*}(A).
    \]
    On the other hand, $\eta_p(d)\rho_{\mathbb W,p,t}(A)\leq\rho_{w,p,t}(A) \leq \kappa_p(d)\rho_{\mathbb W,p,t}(A)$ gives that
    \[
        \kappa_p^{-1}(d)\rho^*_{\mathbb W,p,t}(A)\leq\rho^*_{w,p,t}(A) \leq \eta^{-1}_p(d)\rho^*_{\mathbb W,p,t}(A),
    \]
    thus for any $I\in\II$, we have
    \[
        \kappa_p^{-1}(d)\rho^*_{\mathbb W,p,I}(A)\leq\rho^*_{w,p,I}(A) \leq \eta^{-1}_p(d)\rho^*_{\mathbb W,p,I}(A).
    \]
    Combining the argument above, we have
    \[
        [w]^{1/p}_{A_p(\NN)}=\sup_{I\in\II}\sup_{0\neq A} \frac{\rho_{w,p,I}^*(A)}{(\rho_{w,p,I})^*(A)}\leq \eta^{-1}_p(d)\kappa_p(d)[\mathbb W]^{1/p}_{A_p(\R;M_{d^2}(\C))}=d^{|\frac{1}{p}-\frac{1}{2}|}[\mathbb W]^{1/p}_{A_p(\R;M_{d^2}(\C))}.
    \]
    Similarly, we have $[w]_{A_p(\NN)}\geq d^{-|1-\frac{p}{2}|}[\mathbb W]_{A_p(\R;M_{d^2}(\C))}$, which gives the estimate as desired.
\end{proof}

\subsection{Averaging operators}

Given $I\in\II$, define the positive linear operator
\begin{equation}\label{eq:matrix-averaging-operator}
    \begin{aligned}
        A_I:\NN &\longrightarrow \NN \\
        f&\longmapsto \langle f\rangle_I\cdot \chi_I
    \end{aligned}
\end{equation}
where $\langle f\rangle_I:=\frac{1}{|I|}\int_I f(t)dt$ denotes the integral average of $f$ over $I$. 

\begin{lemma}\label{lem:one-average}
Let $w$ be a matrix weight. Given $I\in\II$, the weighted norm of the averaging operator $A_I$ is 
\begin{equation}\label{eq:one-average-norm}
 \|A_I\|_{L_p(\NN,w)\to L_p(\NN,w)}
 =\sup_{A\ne0}\frac{\rho_{w,p,I}^*(A)}{(\rho_{w,p,I})^*(A)},
\end{equation}
Consequently,
\begin{equation}\label{eq:single-averages-characteristic}
 \sup_{I\in\II}\|A_I\|_{L_p(\NN,w)\to L_p(\NN,w)}
 =[w]_{A_p(\NN)}^{1/p}.
\end{equation}
If $w\notin A_p(\NN)$, we may regard $[w]_{A_p(\NN)}$ as infinity.
\end{lemma}

\begin{proof}
We first consider an operator
\[
    \begin{aligned}
        P_I: L_p(\NN)&\longrightarrow (M_d(\C),\rho_{w,p,I}) \\
        f&\longmapsto \langle w^{-1/2p}fw^{-1/2p}\rangle_I
    \end{aligned}
\]
then
\[
    \begin{aligned}
        \rho_{w,p,I}(P_I f)&=\Big(\frac{1}{|I|}\int_I \| w^{1/2p}(t)\langle w^{-1/2p} fw^{-1/2p}\rangle_I w^{1/2p}(t)\|_{S_p^d}^pdt\Big)^{1/p}\\
        &=|I|^{-1/p}\| A_I (w^{-1/2p}fw^{-1/2p})\|_{L_p(\NN,w)}.
    \end{aligned}
\]
Thus
\[
    \begin{aligned}
        \|P_I\|_{L_p(\NN)\to (M_d(\C),\rho_{w,p,I})}&=\sup_{0\neq f\in L_p(\NN)}\frac{ \rho_{w,p,I}(P_I f)}{\| f\|_{L_p(\NN)}}\\&=|I|^{-1/p}\cdot\sup_{0\neq g\in L_p(\NN,w)}\frac{\| A_Ig\|_{L_p(\NN,w)}}{\| g\|_{L_p(\NN,w)}}.
    \end{aligned}
\]
Thus we only need to calculate the operator norm of $P_I$. The adjoint operator of $P_I$ can be calculated as
\[
    \begin{aligned}
        P_I^*:(M_d(\C),(\rho_{w,p,I})^*)&\longrightarrow L_{p'}(\NN) \\
        A&\longmapsto w^{-1/2p}Aw^{-1/2p}\cdot \frac{\chi_I}{|I|}
    \end{aligned}
\]
Then
\[
    \|A_I\|_{L_p(\NN,w)\to L_p(\NN,w)}=|I|^{1/p}\cdot \sup_{A\neq 0}\frac{\|P_I^*A\|_{L_{p'}(\NN)}}{(\rho_{w,p,I})^*(A)}=\sup_{A\neq 0}\frac{\rho_{w,p,I}^*(A)}{(\rho_{w,p,I})^*(A)}.
\]
Taking the supremum over $I$ gives \eqref{eq:single-averages-characteristic}.
\end{proof}

\subsection{The noncommutative weighted Hardy-Littlewood maximal inequality}
Our proof of the {\it only if} part of Theorem \ref{thm:intro-maximal} needs a comparison between the noncommutative maximal norm and the vector-valued maximal norm.

\begin{lemma}\label{lem:positive-max-to-classical}
Let $\mathcal F$ be finite and
$(h_\alpha)_{\alpha\in\mathcal F}\subset L_p(\NN)$.  Then
\begin{equation}\label{eq:positive-max-to-classical}
 \big\|\max_{\alpha\in\mathcal F}
       \|h_\alpha(\,\cdot\,)\|_{S_p^d}\big\|_{L_p(\R)}
 \le\big\|{\sup_{\alpha\in\mathcal F}}^+h_\alpha\big\|_{L_p(\NN)}
 \le d^{1/p}\big\|\max_{\alpha\in\mathcal F}
       \|h_\alpha(\,\cdot\,)\|_{S_p^d}\big\|_{L_p(\R)}.
\end{equation}
\end{lemma}

\begin{proof}
Set $\widetilde{h}(t)=\max_\alpha\|h_\alpha(t)\|_{S_p^d}$.  For every common
factorization $h_\alpha=ay_\alpha b$, pointwise operator norms followed by
scalar H\"older give
\[
 \|\widetilde{h}\|_{L_p(\R)}
 \le\|a\|_{2p}\sup_\alpha\|y_\alpha\|_\infty\|b\|_{2p}.
\]
Taking the infimum proves the first inequality.  Conversely, on the set $E:=\{x\in \R:\widetilde{h}(x)>0\}$, write
\[
 h_\alpha=(\widetilde{h}^{1/2}\cdot 1_d)(\widetilde{h}^{-1}h_\alpha)(\widetilde{h}^{1/2}\cdot 1_d),
 \qquad \|\widetilde{h}^{-1}h_\alpha\|_{S_\infty^d}\le1,
\]
and set the middle factor equal to zero on $E^c$.  This common factorization gives the second inequality since $\|\widetilde{h}^{1/2}\cdot 1_d\|_{2p}^2=d^{1/p}\|\widetilde{h}\|_{L_p(\R)}$.
\end{proof}

\begin{remark}
The preceding lemma shows that the noncommutative maximal norm
dominates the scalar $L_p$-norm obtained by first taking pointwise
Schatten $p$-norms and then the supremum. In finite matrix dimension,
the reverse comparison holds with the optimal factor $d^{1/p}$.
Indeed, let $h_j=\mathbf{1}_{(0,1)}e_{jj}$ for $1\le j\le d$,
where $(e_{ij})$ are the standard matrix units. Then
\[
 \left\|\max_{1\le j\le d}
       \|h_j(\,\cdot\,)\|_{S_p^d}\right\|_{L_p(\R)}
 =1,
 \qquad
 \big\|{\sup_{1\le j\le d}}^+h_j\big\|_{L_p(\NN)}
 =d^{1/p}.
\]
Thus the reverse estimate cannot hold with a dimension-independent
constant. For $1\le p<\infty$, the same construction with countably
many mutually orthogonal rank-one projections in $B(\ell_2)$,
equipped with its usual trace, gives a scalar maximal function of
$L_p$-norm $1$, but an infinite noncommutative maximal norm.
\end{remark}

We are now ready to prove Theorem \ref{thm:intro-maximal}.

\begin{proof}[Proof of Theorem~\ref{thm:intro-maximal}]
The {\it only if} part. Let $w\in A_p(\NN)$. We aim to prove the maximal inequality
\begin{align*}
    \big\| {\sup_{I}}^+ A_If\big\|_{L_p(\NN,w)}\leq d^{\frac{1}{p}+(1+p')|\frac{1}{p}-\frac{1}{2}|}\mathcal{G}_{p,d^2}[w]_{A_p(\NN)}^{p'/p}\| f\|_{L_p(\NN,w)},\qquad f\in \SS_\NN^c.
\end{align*}
Given $f\in \SS_\NN^c$, set  
$F=w^{1/2p}fw^{1/2p}$. Then we have
\[
 M_{\mathbb W}F(t)
 =\sup_{I\ni t}\frac1{|I|}\int_I
 \|\mathbb W(t)^{1/p}\mathbb W(s)^{-1/p}F(s)\|_{S_2^d}\,ds,
\]
where $\mathbb W$ is related to $w$ by \eqref{equ:mathbb W}.
Lemma~\ref{lem:quantitative goldberg} and Lemma~\ref{lem:compare N and ell_2^d2} give
\begin{equation}\label{eq:IM-quantitative-bound}
 \begin{aligned}
     \|M_{\mathbb W}F\|_{L_p(\R)}
 &\leq \GG_{p,d^2}[\mathbb W]_{A_p(\R;M_{d^2}(\C))}^{p'/p}
 \|F\|_{L_p(\R;S_2^d)} \\
 &\leq d^{\max\{0,\frac{1}{2}-\frac{1}{p}\}+p'|\frac{1}{2}-\frac{1}{p}|}\GG_{p,d^2} [w]_{A_p(\NN)}^{p'/p}
 \|F\|_{L_p(\R;S_p^d)}
 \end{aligned}
\end{equation}
By Minkowski's inequality and
\eqref{eq:operator-weight-powers},
\[
 \begin{aligned}
     \sup_{I\in \II}\|w(t)^{1/2p}A_If(t)w(t)^{1/2p}\|_{S_p^d}&=\sup_{I\ni t}\frac{1}{|I|}\Big\|\int_I w(t)^{1/2p}f(s)w(t)^{1/2p}ds\Big\|_{S_p^d} \\
 &\leq d^{\max\{0,\frac{1}{p}-\frac{1}{2}\}}\sup_{I\ni t}\frac{1}{|I|}\int_I \| \mathbb W^{1/p}(t)\mathbb W^{-1/p}(s)F(s)\|_{S_2^d}ds\\
 &= d^{\max\{0,\frac{1}{p}-\frac{1}{2}\}}M_{\mathbb W}F(t).
 \end{aligned}
\]
Thus, for every finite $\mathcal F\subset\II$, Lemma
\ref{lem:positive-max-to-classical}, \eqref{eq:IM-quantitative-bound},
and equivalence of Schatten norms yield
\[
\begin{aligned}
    \big\|{\sup_{I\in\mathcal F}}^+A_If\big\|_{L_p(\NN,w)}&\leq d^{1/p}\big\| \max_{I\in \mathcal F}\|w^{1/2p}(A_I f)w^{1/2p}\|_{S_p^d}\big\|_{L_p(\R)} \\
 &\leq d^{\frac{1}{p}+(1+p')|\frac{1}{p}-\frac{1}{2}|}\mathcal{G}_{p,d^2}[w]_{A_p(\NN)}^{p'/p}
 \|f\|_{L_p(\NN,w)}.
\end{aligned}
\]
The finite-subfamily formula \eqref{eq:prelim-finite-subfamilies} gives the
full interval family, and density of $\SS_\NN^c$ from
Lemma~\ref{lem:density of N} gives the bounded extension.  This proves the desired maximal inequality.

The {\it if} part. If the maximal operator has finite norm $C_{p,d}(w)$, every coordinate $A_I$ has
norm at most $C_{p,d}(w)$.  Lemma~\ref{lem:one-average} therefore gives
\[
 [w]_{A_p(\NN)}^{1/p}
 =\sup_{I\in\II}\|A_I\|_{L_p(\NN,w)\to L_p(\NN,w)}
 \le C_{p,d}(w),
\]
which proves the lower bound and completes the proof.
\end{proof}

\section{{A noncommutative analogue of Rubio de Francia's abstract extrapolation theory}}\label{sec:abstract Rubio}

{In this section, we establish a noncommutative version of Rubio de Francia's abstract extrapolation theory. For technical reasons, we can only work on trace-finite von Neumann algebras at the moment of writing. Let $(\M,\tau)$ be trace-finite and let $(E_\beta)_{\beta\in\BB}$ be a family of positive linear maps on $\M$. We will first introduce the weight class associated to $(E_\beta)_{\beta\in\BB}$, then provide a factorization result for these weight classes, and finally establish the extrapolation theorem for operators.}

\subsection{The weight classes and factorization}

\begin{definition}\label{def:weight-classes}
Let $1<p<\infty$.  We write $w\in W_p((E_\beta)_\beta)$ if
$w\in L_0^+(\M)$ is an operator weight and
\[
 x\longmapsto(E_\beta(x))_{\beta\in\BB},
 \qquad x\in\mathscr \M\cap L_p(\M,w),
\]
extends boundedly from $L_p(\M,w)$ to
$L_p(\M,w;\ell_\infty(\BB))$.  The norm of this extension is denoted by
$[w]_{W_p((E_\beta)_\beta)}$, or simply $[w]_{W_p}$; equivalently,
\begin{equation}\label{eq:Wp-def}
 \big\|{\sup_{\beta\in\BB}}^+E_\beta(x)\big\|_{L_p(\M,w)}
 \le [w]_{W_p}\|x\|_{L_p(\M,w)},\qquad x\in L_p(\M,w).
\end{equation}
\end{definition}

The extension is unique because $\M\cap L_p(\M,w)\subset L_p(\M,w)$ from Lemma~\ref{lem:density of M} is dense.  

\begin{assumption}\label{ass:symmetric}
For every $1<p<\infty$, $w\in W_p((E_\beta)_\beta)$ is equivalent to $w^{-p'/p}\in W_{p'}((E_\beta)_\beta)$. Quantitatively, there is a fixed increasing function
$\gamma_p:[0,\infty)\to[0,\infty)$ such that
\begin{equation}\label{eq:gamma-control}
 [w^{-p'/p}]_{W_{p'}}
 \le
 \gamma_p\bigl([w]_{W_p}\bigr).
\end{equation}
\end{assumption}

Applying the same assumption at $p'$ to $w^{-p'/p}$ also gives the
reverse set inclusion and quantitative control. 

For $1<p,r<\infty$, define
\begin{equation}\label{eq:quantitative-control-functions}
 \Theta_{p,r}(t)=
 \begin{cases}
 \displaystyle
 \gamma_{r'}\!\left(
    4^{1-p'/r'}\max\{1,\gamma_p(t)\}
 \right),&1<r<p,\\[0.8em]
 \displaystyle
 4^{1-p/r}\max\{1,t\},&1<p\leq r.
 \end{cases}
\end{equation}

The following equivalence statement will play an instrumental role in showing a factorization result.

\begin{lemma}\label{lem:tech}
Let $1<p<\infty$ and $\mathbf{w} \in L_0(\M)$ be invertible. Then $|\mathbf{w}|^{2p}\in W_p((E_\beta)_\beta)$ if and only if
\begin{equation}\label{eq:polar-transfer}
 \big\|
{\sup_{\beta\in\BB}}^+\mathbf{w}E_\beta\bigl(\mathbf{w}^{-1}x(\mathbf{w}^*)^{-1}\bigr)\mathbf{w}^* \big\|_p \le C\|x\|_p,\qquad x\in L_p(\M).
\end{equation}
The least constant in \eqref{eq:polar-transfer} equals
$[|\mathbf{w}|^{2p}]_{W_p((E_\beta)_\beta)}$.
\end{lemma}

\begin{proof}
Firstly, we claim that the map
\[
    \begin{aligned}
        \iota:L_p(\M,|\mathbf{w}|^{2p})&\longrightarrow L_p(\M) \\
        x&\longmapsto \mathbf{w}x\mathbf{w}^*
    \end{aligned}
\]
is an onto isometry. Indeed, on the one hand,
\[
    \| x\|_{L_p(\M,|\mathbf{w}|^{2p})}=\| |\mathbf{w}|\cdot x\cdot|\mathbf{w}|\|_p=\| \mathbf{w}x\mathbf{w}^*\|_p=\| \iota(x)\|_p,
\]
which shows that $\iota$ is an isometry. On the other hand, given $x\in L_p(\M)$, then it is clear that $\iota(\mathbf{w}^{-1}x(\mathbf{w}^*)^{-1})=x$. Thus \eqref{eq:polar-transfer} is equivalent to
\[
    \big\| {\sup_{\beta\in\BB}}^+E_\beta(y) \big\|_{L_p(\M,|\mathbf{w}|^{2p})}=\big\| {\sup_{\beta\in\BB}}^+\mathbf{w}E_\beta(y)\mathbf{w}^* \big\|_p \le C\|y\|_{L_p(\M,|\mathbf{w}|^{2p})},\qquad y\in L_p(\M,|\mathbf{w}|^{2p}),
\]
{which gives the estimate as desired.}
\end{proof}

\begin{theorem}\label{thm:factorization-both-sides}
Suppose that Assumption~\ref{ass:symmetric} holds.  Let
$1<p,r<\infty$, $p\ne r$, and define $s>1$ by
\(
 1/s=|1/p-1/r|.
 \)
If $w\in W_p((E_\beta)_\beta)$ and
$u\in L_s^+(\M)$ has full support, then there is a full-support operator $v\in L_s^+(\M)$ such that
\begin{equation}\label{eq:factorization-majorant}
 u\le v,
 \qquad \|v\|_s\le2\|u\|_s,
\end{equation}
and
\begin{equation}\label{eq:factorization-output-weights}
\begin{cases}
 |v^{1/2}w^{1/2p}|^{2r}\in W_r((E_\beta)_\beta),&1<r<p,\\
 |v^{-1/2}w^{1/2p}|^{2r}\in W_r((E_\beta)_\beta),&1<p<r.
\end{cases}
\end{equation}
Quantitatively,
\begin{equation}\label{eq:factorization-characteristic-control}
\begin{aligned}
 [|v^{1/2}w^{1/2p}|^{2r}]_{W_r}&\le
 \gamma_{r'}\!\left(
    4^{1-p'/r'}
    \max\left\{1,
       \gamma_p\bigl([w]_{W_p}\bigr)
    \right\}
 \right),&&1<r<p,\\[0.5em]
 [|v^{-1/2}w^{1/2p}|^{2r}]_{W_r}
 &\le4^{1-p/r}\max\{1,[w]_{W_p}\},
 &&1<p<r.
\end{aligned}
\end{equation}
\end{theorem}

\begin{proof}
Firstly, if $v$ is full-support and $w$ is an operator weight, then both $|v^{1/2}w^{1/2p}|^{2r}$ and $|v^{-1/2}w^{1/2p}|^{2r}$ are operator weights.

For the proof, denote by $S_\beta$ the positive linear operator on
$L_p(\M)$ defined as
\begin{equation}\label{eq:transported-structural-map}
 S_\beta(x)
 =w^{1/2p}E_\beta(w^{-1/2p}xw^{-1/2p})w^{1/2p}.
\end{equation}
Definition~\ref{def:weight-classes} gives
\begin{equation}\label{eq:transported-p-bound}
 \big\|{\sup_{\beta\in\BB}}^+S_\beta(x)\big\|_{p}
 \le C\|x\|_p,
 \qquad
 C=\max\{1,[w]_{W_p}\}.
\end{equation}

We first establish an endpoint construction.  Let $h\in L_p^+(\M)$ have
full support.  Starting with $h_0=h$, the positive-majorant formula
\eqref{eq:prelim-positive-majorant} allows us to choose inductively
$h_{k+1}\in L_p^+(\M)$ such that
\[
 S_\beta(h_k)\le h_{k+1}\quad(\beta\in\BB),
 \qquad
 \|h_{k+1}\|_p\le2C\|h_k\|_p.
\]
Consequently,
\begin{equation}\label{eq:endpoint-majorant}
 v_0=\sum_{k=0}^\infty(4C)^{-k}h_k
\end{equation}
converges in $L_p^+(\M)$ and satisfies
\begin{equation}\label{eq:v0-properties}
 h\le v_0,
 \qquad \|v_0\|_p\le2\|h\|_p,
 \qquad S_\beta(v_0)\le4Cv_0.
\end{equation}
The last inequality follows by applying the positive map
$S_\beta$ term by term to the convergent series.

For $x\in\M$, consider the positive maps
\[
 x\longmapsto
 v_0^{-1/2}S_\beta(v_0^{1/2}xv_0^{1/2})v_0^{-1/2}.
\]
If $0\le x\le1$, \eqref{eq:v0-properties} bounds the image by $4C1$.
Thus every coordinate is a bounded positive map on $\M$ with norm at most
$4C$.  Since
$L_\infty(\M;\ell_\infty(\BB))=\ell_\infty(\BB;\M)$ isometrically, this
gives, for arbitrary $x\in\M$,
\begin{equation}\label{eq:transported-infinity-bound}
    \big\|{\sup_{\beta\in\BB}}^+v_0^{-1/2}S_\beta(v_0^{1/2}xv_0^{1/2})v_0^{-1/2}\big\|_{\infty}\le4C\|x\|_\infty.
\end{equation}

Suppose now that $p<r$.  Then $1/s=1/p-1/r$.  Apply the endpoint
construction to $h=u^{s/p}$.  Its assumptions are satisfied because $u$
is positive with full support and
$\|u^{s/p}\|_p=\|u\|_s^{s/p}$.  Put
\[
 \theta=\frac ps=1-\frac pr,
 \qquad v=v_0^\theta.
\]
The L\"owner--Heinz inequality for positive operators
\cite{Kato1952}, applied with $0<\theta<1$ to
$u^{s/p}\le v_0$, gives
\begin{equation}\label{eq:v-from-v0}
 u\le v,
 \qquad
 \|v\|_s=\|v_0\|_p^\theta\le2^\theta\|u\|_s\le2\|u\|_s.
\end{equation}

We include the truncation needed for complex interpolation.  Let
$e_k=1_{[1/k,k]}(v_0)$, fix a finite
$\mathcal F\subset\BB$, and, for $m,n\ge1$, set
\begin{equation}\label{eq:truncated-analytic-family}
\begin{aligned}
 \mathscr T_{m,n}(z)x
 =\Big(&e_mv_0^{-z/2}
 S_\beta(v_0^{z/2}e_nxe_nv_0^{z/2})
 v_0^{-z/2}e_m\Big)_{\beta\in\mathcal F}.
 \qquad0\le\Re z\le1.
\end{aligned}
\end{equation}
{The spectral projections make sure that $\mathscr T_{m,n}$ is a well-defined analytic operator from the dense class $\M$ to $L_p(\M;\ell_\infty(\mathcal F))$.} 

{By \eqref{eq:transported-p-bound} and \eqref{eq:transported-infinity-bound}}, we have the following estimates on the lines $\Re z=0, 1$:
\begin{equation}\label{equ:abs-p}
\begin{aligned}
\|\mathscr T_{m,n}(it)x\|_{L_p(\mathcal M;\ell_\infty)}
&\le
\Big\|
\big(S_\beta(v_0^{it/2}e_nxe_nv_0^{it/2})\big)_{\beta\in\mathcal F}
\Big\|_{L_p(\mathcal M;\ell_\infty)} \\
&\le C\|v_0^{it/2}e_nxe_nv_0^{it/2}\|_p
\le C\|x\|_p.
\end{aligned}
\end{equation}
\begin{equation}\label{equ:abs-infty}
\begin{aligned}
\|\mathscr T_{m,n}(1+it)x\|_{L_\infty(\mathcal M;\ell_\infty)}
&\le4C\|v_0^{it/2}e_nxe_nv_0^{it/2}\|_\infty \\
&=4C\|e_nxe_n\|_\infty
\le4C\|x\|_\infty.
\end{aligned}
\end{equation}
Stein's interpolation \cite{Stein1956}, together with the interpolation
identity in Lemma~\ref{lem:Lp-l1-interpolation}, gives at
$\theta=1-p/r$
\begin{equation}\label{eq:truncated-center-estimate}
 \|\mathscr T_{m,n}(\theta)x\|_{L_r(\M;\ell_\infty)}
 \le4^\theta C\|x\|_r.
\end{equation}

Letting $m\to\infty$ in \eqref{eq:truncated-center-estimate} gives that
\[
    \big\|{\sup_{\beta\in\mathcal F}}^+v^{-1/2}S_\beta(v^{1/2}e_nxe_nv^{1/2})v^{-1/2}\big\|_{r}\le4^\theta C\|x\|_r,\qquad x\in \M.
\]
For $x\in \mathscr C(v_0)$, there exists some $n\in \N_+$ such that $e_nxe_n=x$, and thus
\[
    \big\|{\sup_{\beta\in\mathcal F}}^+v^{-1/2}S_\beta(v^{1/2}xv^{1/2})v^{-1/2}\big\|_{r}\le4^\theta C\|x\|_r,\qquad x\in \mathscr C(v_0).
\]
The density of $\mathscr C(v_0)\subset L_r(\M)$ in Lemma~\ref{lem:density of M} and the finite-subfamily property give that
\begin{align*}    \big\|{\sup_{\beta\in\BB}}^+v^{-1/2}S_\beta(v^{1/2}xv^{1/2})v^{-1/2}\big\|_{r}\le4^{1-p/r} C\|x\|_r,\qquad x\in  L_r(\M).
\end{align*}
{Then Lemma~\ref{lem:tech}, applied to $\mathbf{w}=v^{-1/2}w^{1/2p}$, gives the case of $p<r$.}

Finally, suppose that $r<p$.  Assumption~\ref{ass:symmetric} gives
\[
 w^{-p'/p}\in W_{p'}((E_\beta)_\beta),
 \qquad
 [w^{-p'/p}]_{W_{p'}}
 \le\gamma_p([w]_{W_p}).
\]
Here $p'<r'$ and
$1/s=1/p'-1/r'$, so the case just proved applies to this dual weight.
It gives the same majorant $v$ and
\[
 |v^{-1/2}w^{-1/2p}|^{2r'}
 \in W_{r'}((E_\beta)_\beta)
\]
with characteristic at most
\(
 4^{1-p'/r'}
 \max\{1,\gamma_p([w]_{W_p})\}.
\)
Applying Assumption~\ref{ass:symmetric} at $r'$ and using measurable
functional calculus give
\begin{equation}\label{eq:dual-factorization-computation}
\begin{aligned}
 \bigl(|v^{-1/2}w^{-1/2p}|^{2r'}\bigr)^{-r/r'}
 &=(w^{-1/2p}v^{-1}w^{-1/2p})^{-r}\\
 &=(w^{1/2p}vw^{1/2p})^r
  =|v^{1/2}w^{1/2p}|^{2r}.
\end{aligned}
\end{equation}
This proves {the case of $r<p$} and exactly the first estimate in
\eqref{eq:factorization-characteristic-control}.
\end{proof}

\subsection{A noncommutative Rubio de Francia extrapolation theorem}
Now we exploit the factorization for the weights to show the extrapolation result for the operators.

\begin{theorem}\label{thm:extra-abstract}
Suppose that Assumption~\ref{ass:symmetric} holds. Let
$(T_\alpha)_{\alpha\in\AA}$ be linear maps from $\M$ to $L_0(\M)$. Suppose, for some $1<r<\infty$ and an increasing function
$\Phi:[0,\infty)\to[0,\infty)$, one has
\begin{equation}\label{eq:abstract-r-hypothesis}
 \big\|{\sup_{\alpha\in\AA}}^+T_\alpha(x)\big\|_{L_r(\M,w)}
 \le
 \Phi\bigl([w]_{W_r}\bigr)
 \|x\|_{L_r(\M,w)}
\end{equation}
for every $x\in \M\cap L_r(\M,w)$ and $w\in W_r((E_\beta)_\beta)$. Then, for every $1<p<\infty$, $w\in W_p((E_\beta)_\beta)$, the map $x\mapsto (T_\alpha x)_\alpha$ extends uniquely to a bounded linear map from $L_p(\M,w)$ to $L_p(\M,w;\ell_\infty(\AA))$, and
\begin{equation}\label{eq:abstract-p-conclusion}
 \big\|{\sup_{\alpha\in\AA}}^+T_\alpha(x)\big\|_{L_p(\M,w)}
 \le
 8\Phi\!\left(
    \Theta_{p,r}\bigl([w]_{W_p}\bigr)
 \right)
 \|x\|_{L_p(\M,w)}.
\end{equation}
\end{theorem}

As mentioned in Section 1, in Theorem~\ref{thm:extra-abstract}, letting $T_{\alpha}\equiv T$, the resulting maximal norm is exactly $\|T(x)\|_{L_p(\M,w)}$. Then we have the following extrapolation theorem for a single linear operator.
\begin{corollary}
Suppose that Assumption~\ref{ass:symmetric} holds. Let $T$ be a linear map from $\M$ to $L_0(\M)$. If, for some $1<r<\infty$ and increasing $\Phi$, one has
\[
 \|T(x)\|_{L_r(\M,w)}
 \le\Phi([w]_{W_r})\|x\|_{L_r(\M,w)}
\]
for every $w\in W_r((E_\beta)_\beta)$ and $x\in \M\cap L_r(\M,w)$, then $T$ extends uniquely to $L_p(\M,w)$ for every $1<p<\infty$ and $w\in W_p((E_\beta)_\beta)$, with
\[
 \|T(x)\|_{L_p(\M,w)}
 \le8\Phi\!\left(
    \Theta_{p,r}\bigl([w]_{W_p}\bigr)
 \right)\|x\|_{L_p(\M,w)}.
\]
\end{corollary}

With the factorization for weights---Theorem \ref{thm:factorization-both-sides}, the H\"older type testing estimate established in Lemma \ref{lem:maximal-testing} yields Theorem~\ref{thm:extra-abstract}.

\begin{proof}[Proof of Theorem~\ref{thm:extra-abstract}]
Fix $w\in W_p((E_\beta)_\beta)$ and
$x\in\M\cap L_p(\M,w)$.

Suppose first that $r<p$, so $1/r=1/p+1/s$.  For an arbitrary positive
invertible $u\in L_s(\M)$, Theorem
\ref{thm:factorization-both-sides} gives $v\ge u$ with
$\|v\|_s\le2\|u\|_s$ and
$|v^{1/2}w^{1/2p}|^{2r}\in W_r((E_\beta)_\beta)$.  The polar
decomposition of $v^{1/2}w^{1/2p}$ and H\"older's inequality give
\[
 \|x\|_{L_r(\M,|v^{1/2}w^{1/2p}|^{2r})}
 =\|v^{1/2}w^{1/2p}xw^{1/2p}v^{1/2}\|_r
 \le\|v\|_s\|x\|_{L_p(\M,w)}<\infty.
\]
That is, $x\in\M\cap L_r(\M,|v^{1/2}w^{1/2p}|^{2r})$. Applying the assumption \eqref{eq:abstract-r-hypothesis}, followed by the monotonicity property---Lemma~\ref{lem:compare-sandwich}, yields
\begin{align*}
 \big\|{\sup_{\alpha\in\AA}}^+
 u^{1/2}w^{1/2p}T_\alpha(x)w^{1/2p}u^{1/2}\big\|_r &\leq \big\|{\sup_{\alpha\in\AA}}^+
 v^{1/2}w^{1/2p}T_\alpha(x)w^{1/2p}v^{1/2}\big\|_r \\
 &=\big\|{\sup_{\alpha\in\AA}}^+
 T_\alpha(x)\big\|_{L_r(\M,|v^{1/2}w^{1/2p}|^{2r})}\\ 
 &\leq\Phi([|v^{1/2}w^{1/2p}|^{2r}]_{W_r})\cdot \| x\|_{L_r(\M,|v^{1/2}w^{1/2p}|^{2r})} \\
 &\leq 2\Phi\!\left(
   \Theta_{p,r}([w]_{W_p})
 \right) \cdot \| u\|_s \cdot \| x\|_{L_p(\M,w)}.
\end{align*}
Divide by $\|u\|_s$, take the supremum over $u$, and use
Lemma~\ref{lem:maximal-testing}(ii).  This gives
\eqref{eq:abstract-p-conclusion}, with the factor $8=4\cdot2$.

Suppose next that $p<r$, so $1/p=1/r+1/s$.  Given $\varepsilon>0$,
Lemma~\ref{lem:maximal-testing}(i), applied to $w^{1/2p}xw^{1/2p}$, gives a positive
invertible $u\in L_s(\M)$ such that
\begin{equation}\label{eq:inverse-test-choice}
 \|u^{-1/2}w^{1/2p}xw^{1/2p}u^{-1/2}\|_r\|u\|_s
 \le2\|x\|_{L_p(\M,w)}+\varepsilon.
\end{equation}
Theorem~\ref{thm:factorization-both-sides} gives $v\ge u$ with
$\|v\|_s\le2\|u\|_s$ and
$|v^{-1/2}w^{1/2p}|^{2r}\in W_r((E_\beta)_\beta)$. Then
 \[
     \begin{aligned}
         \|x\|_{L_r(\M,|v^{-1/2}w^{1/2p}|^{2r})}&=\| v^{-1/2}w^{1/2p}xw^{1/2p}v^{-1/2}\|_r\\ &\leq \| u^{-1/2}w^{1/2p}xw^{1/2p}u^{-1/2}\|_r \leq \frac{1}{\| u\|_s}\cdot (2\| x\|_{L_p(\M,w)}+\varepsilon)<\infty
     \end{aligned}
 \]
 gives that $x\in \M\cap L_r(\M,|v^{-1/2}w^{1/2p}|^{2r})$. Again the H\"older type testing estimate---Lemma~\ref{lem:maximal-testing}(ii) and the assumption \eqref{eq:abstract-r-hypothesis} give
\begin{align*}
 \big\|{\sup_{\alpha\in\AA}}^+T_\alpha(x)\big\|_{L_p(\M,w)}
 &\le \|v\|_s
 \big\|{\sup_{\alpha\in\AA}}^+
 v^{-1/2}w^{1/2p}T_\alpha(x)w^{1/2p}v^{-1/2}\big\|_r\\
 &\leq 2\| u\|_s\cdot \big\|{\sup_{\alpha\in\AA}}^+
 T_\alpha(x)\big\|_{L_r(\M,|v^{-1/2}w^{1/2p}|^{2r})} \\
 &\le2\Phi\big(
   [|v^{-1/2}w^{1/2p}|^{2r}]_{W_r}
 \big)
 \|u\|_s\cdot \| x\|_{L_r(\M,|v^{-1/2}w^{1/2p}|^{2r})}\\
 &\le2\Phi\big(
   \Theta_{p,r}([w]_{W_p})
 \big)\cdot (2\| x\|_{L_p(\M,w)}+\varepsilon).
\end{align*}
Letting $\varepsilon\downarrow0$ gives the sharper factor $4$ in this
case, and hence the stated uniform factor $8$.
\end{proof}

\section{Matrix-weighted extrapolation and applications}\label{sec:extrapolation on N}

In this section, we prove Theorem~\ref{thm:intro-extrapolation} by applying the argument of
Section~4 directly to interval averages on $\NN$.  We first verify the
structural duality assumption and record the changes needed for factorization, then
give the quantitative control and applications.  

\subsection{The factorization structure on \texorpdfstring{$\NN$}{N}}Recall that $\mathcal N=L_\infty(\mathbb R)\bar\otimes{M_d(\mathbb C)}$ is equipped with the tensor trace, and that $(A_{I})_{I\in\mathcal I}$ is the family of averaging operators.
According to the notation in Section \ref{sec:abstract Rubio}, for a matrix weight $w$, we write $w\in W_p((A_I)_I)$ if
\[
 f\longmapsto(A_If)_{I\in\II},
 \qquad f\in \SS_\NN^c,
\]
extends boundedly from $L_p(\NN,w)$ to
$L_p(\NN,w;\ell_\infty)$.  Its operator norm is denoted by $[w]_{W_p((A_I)_I)}$. And by Theorem~\ref{thm:intro-maximal}, we have the quantitative estimate
\begin{equation}\label{equ:quantitative Wp Ap}
[w]_{A_p(\NN)}^{1/p}
 \le [w]_{W_p((A_I)_I)}
 \le \WW_{p,d}[w]_{A_p(\NN)}^{p'/p},
\end{equation}
with $\WW_{p,d}\leq d^{\frac{1}{p}+(1+p')|\frac{1}{p}-\frac{1}{2}|}\mathcal{G}_{p,d^2}$.

\begin{lemma}\label{lem:N-Wp-symmetry}
The class $W_p((A_I)_I)$, where $1<p<\infty$, satisfies
Assumption~\ref{ass:symmetric}.  More precisely, one may take
\begin{equation}\label{eq:N-gamma}
 \gamma_p(t)=\gamma_{p,d}(t)=\WW_{p',d}t^p.
\end{equation}
\end{lemma}

\begin{proof}
Let $w\in W_p((A_I)_I)$.  By definition $w$ is a matrix weight, and
Theorem~\ref{thm:intro-maximal} gives $w\in A_p(\NN)$.  Proposition
\ref{prop:matrix-Ap-dual} then gives
$w^{-p'/p}\in A_{p'}(\NN)$ and the exact identity
\eqref{eq:Ap-dual-characteristic}.  Applying the upper comparison in
Theorem~\ref{thm:intro-maximal} at $p'$, and then its lower comparison
at $p$, yields
\begin{align*}
 [w^{-p'/p}]_{W_{p'}((A_I)_I)}
 &\le \WW_{p',d}
       [w^{-p'/p}]_{A_{p'}(\NN)}^{p/p'}\\
 &=\WW_{p',d}[w]_{A_p(\NN)}\\
 &\le \WW_{p',d}[w]_{W_p((A_I)_I)}^p.
\end{align*}
Geometric duality also ensures local integrability of the inverse power.
The same argument at $p'$ gives the reverse implication.
\end{proof}

\begin{remark}\label{most important remark}
The realization of Assumption~\ref{ass:symmetric} on $\NN$ established above uses the control function
\[
\gamma_{p,d}(t)=\WW_{p',d}t^p,
\]
which depends on the matrix dimension $d$. It is then natural to ask for its sharp dependence with respect to $d$. It is even more ambitious to ask whether $\gamma_{p,d}$ can be taken to be independent of $d$. A positive answer would lead to a truly noncommutative weighted and extrapolation theory. 

\end{remark}

The factorization theorem now transfers to $\NN$. Compared with the proof of Theorem~\ref{thm:factorization-both-sides}, there are two different points in this semifinite case. Firstly, we need to verify that the new functions we get are matrix weights, i.e., we have to check that both $|v^{1/2}w^{1/2p}|^{2r}$ ($r<p$) and $|v^{-1/2}w^{1/2p}|^{2r}$ ($p<r$) are locally integrable. Secondly, thanks to Lemma~\ref{lem:one-average}, we can get a better quantitative estimate in \eqref{eq:factorization-characteristic-control-NN} than by simply combining \eqref{eq:factorization-characteristic-control} and \eqref{equ:quantitative Wp Ap}.

\begin{theorem}\label{thm:N-factorization}
Let $1<p,r<\infty$, $p\ne r$, and define $s>1$ by \(1/s=|1/p-1/r|\). If $w\in A_p(\NN)$ and $u\in L_s^+(\NN)$ is an invertible function, then there is an invertible function $v\in L_s^+(\NN)$ such that
\begin{equation}\label{eq:factorization-majorant-NN}
 u\le v, \qquad \|v\|_s\le2\|u\|_s,
\end{equation}
and
\begin{equation}\label{eq:factorization-output-weights-NN}
\begin{cases}
 |v^{1/2}w^{1/2p}|^{2r}\in A_r(\NN),&1<r<p,\\
 |v^{-1/2}w^{1/2p}|^{2r}\in A_r(\NN),&1<p<r.
\end{cases}
\end{equation}
Quantitatively,
\begin{equation}\label{eq:factorization-characteristic-control-NN}
\begin{aligned}
 [|v^{1/2}w^{1/2p}|^{2r}]_{A_r(\NN)}
 &\le
 (4\WW_{p',d})^{(p-r)/(p-1)}
 [w]_{A_p(\NN)},
 &&1<r<p,\\[0.5em]
 [|v^{-1/2}w^{1/2p}|^{2r}]_{A_r(\NN)}
 &\le
 (4\WW_{p,d})^{r-p}
 [w]_{A_p(\NN)}^{(r-1)/(p-1)},
 &&1<p<r,
\end{aligned}
\end{equation}
where $\WW_{p,d}$ is the constant in \eqref{equ:quantitative Wp Ap}.
\end{theorem}

\begin{proof}

    On the one hand, suppose first that $p<r$, and let us verify that $|v^{-1/2}w^{1/2p}|^{2r}$ is locally integrable. Given a bounded interval $J\in\II$, set
    \[
        E_n:=(-n,n)\cap\bigl\{t:\|(v^{-1/2}w^{1/2p})(t)\|_{S_\infty^d}\le n, \|(w^{-1/2p}v^{1/2})(t)\|_{S_\infty^d}\le n \bigr\},
    \]
    and
    \[
        f_n:=\chi_{J\cap E_n}\, v^{-1/2}w^{1/p}v^{-1/2}.
    \]
    Then $f_n$ is bounded and compactly supported, so $f_n\in\SS_\NN^c\subset L_r(\NN)$. According to the proof of Theorem~\ref{thm:factorization-both-sides}, we have the following estimate:
    \[
    \big\|{\sup_{I\in\II}}^+ v^{-1/2}w^{1/2p} A_I(w^{-1/2p}v^{1/2}fv^{1/2}w^{-1/2p}) w^{1/2p}v^{-1/2} \big\|_r\lesssim \|f\|_r, \qquad f\in\SS_\NN^c.
    \]
    In particular, we have
    \begin{equation}\label{equ:single AJ}
    \big\| v^{-1/2}w^{1/2p} A_J(w^{-1/2p}v^{1/2}f_nv^{1/2}w^{-1/2p}) w^{1/2p}v^{-1/2} \big\|_r\lesssim \|f_n\|_r<\infty.
    \end{equation}
    The left hand side of \eqref{equ:single AJ} is exactly
    \[
 \frac{|J\cap E_n|}{|J|} \left( \int_J \tr\bigl(|v^{-1/2}w^{1/2p}|^{2r}(t)\bigr)\,dt \right)^{1/r}
    \]
    Taking $n$ large enough so that $|J\cap E_n|>0$ gives the local integrability. If $r<p$, combining Proposition~\ref{prop:matrix-Ap-dual} and the argument above gives the local integrability as desired.

    On the other hand, we improve the quantitative estimate. For $p<r$, repeat the procedure in the proof of Theorem~\ref{thm:factorization-both-sides} to construct the factorization $v$. Given $I\in \II$, we consider the index set $\mathcal{F}_I=\{ I\}$ of one single element. Lemma~\ref{lem:one-average} gives the estimate on $\Re z=0$:
    \[
    \begin{aligned}
        &\|\mathscr T_{m,n}(it)f\|_{L_p(\NN;\ell_\infty(\mathcal F_I))}\\=&\| e_m v_0^{-it/2} w^{1/2p}A_I(w^{-1/2p}v_0^{it/2}e_nfe_nv_0^{it/2}w^{-1/2p})w^{1/2p}v_0^{-it/2}e_m\|_p\leq [w]_{A_p(\NN)}^{1/p}\|f\|_p.
    \end{aligned}
    \]
    Similar to the proof of \eqref{equ:abs-infty}, we have the estimate on $\Re z=1$:
    \[
    \begin{aligned}
        \|\mathscr T_{m,n}(1+it)f\|_{L_\infty(\NN;\ell_\infty(\mathcal F_I))}\leq 4[w]_{W_p((A_I)_I)}\|f\|_\infty\leq 4\WW_{p,d}[w]_{A_p(\NN)}^{1/(p-1)}\|f\|_\infty,
    \end{aligned}
    \]
    where the last inequality comes from \eqref{equ:quantitative Wp Ap}.
    Interpolation at $\theta=1-p/r$ therefore yields
    \[
    \|\mathscr T_{m,n}(\theta)f\|_r \le [w]_{A_p(\NN)}^{1/r} \big(4\WW_{p,d}[w]_{A_p(\NN)}^{1/(p-1)}\big)^{1-p/r}\|f\|_r.
    \]
    Passing to the cutoff limits as in the proof of Theorem~\ref{thm:factorization-both-sides}, and using Lemma~\ref{lem:tech} and Lemma~\ref{lem:one-average} again, gives
    \[
    \begin{aligned}
    [|v^{-1/2}w^{1/2p}|^{2r}]_{A_r(\NN)} \le [w]_{A_p(\NN)}\big(4\WW_{p,d}[w]_{A_p(\NN)}^{1/(p-1)}\big)^{r-p} = (4\WW_{p,d})^{r-p} [w]_{A_p(\NN)}^{(r-1)/(p-1)}.
    \end{aligned}
    \]
    If $r<p$, then 
    \[
    \begin{aligned}
    [|v^{1/2}w^{1/2p}|^{2r}]_{A_r(\NN)}&= [|v^{-1/2}w^{-1/2p}|^{2r'}]_{A_{r'}(\NN)}^{r/r'}\\
    &\le \left( (4\WW_{p',d})^{r'-p'} [w^{-p'/p}]_{A_{p'}(\NN)}^{(r'-1)/(p'-1)} \right)^{r/r'}\\
    &= (4\WW_{p',d})^{(p-r)/(p-1)} [w]_{A_p(\NN)},
    \end{aligned}
    \]
    where the first and last equalities use \eqref{eq:Ap-dual-characteristic}. This proves \eqref{eq:factorization-characteristic-control-NN}.
\end{proof}

\begin{remark}
    The quantitative estimate \eqref{eq:factorization-characteristic-control-NN} is better than the direct combination of \eqref{eq:factorization-characteristic-control} and \eqref{equ:quantitative Wp Ap}. Indeed, \eqref{eq:factorization-characteristic-control} and \eqref{equ:quantitative Wp Ap} give
    \[
    \begin{aligned}
        [|v^{1/2}w^{1/2p}|^{2r}]_{A_r(\NN)}&\leq 4^{r(r'-p')}\WW^r_{r,d}\WW_{p',d}^{rr'}\WW_{p,d}^{prr'}[w]_{A_p(\NN)}^{prr'/(p-1)},\qquad 1<r<p
    \end{aligned}
    \]
    \[
    \begin{aligned}
        [|v^{-1/2}w^{1/2p}|^{2r}]_{A_r(\NN)}&\leq 4^{r-p}\WW_{r,d}\WW_{p,d}^{r}[w]_{A_p(\NN)}^{r/(p-1)},\qquad 1<p<r.
    \end{aligned}
    \]
    And it is easily verified that $prr'/(p-1)\geq 1$ when $1<r<p$ and $r/(p-1)\geq (r-1)/(p-1)$ when $1<p<r$.
\end{remark}

\subsection{Matrix-weighted extrapolation}

We are now ready to prove Theorem~\ref{thm:intro-extrapolation}.

\begin{proof}[Proof of Theorem~\ref{thm:intro-extrapolation}]
Let $w\in A_p(\NN)$ and $f\in\SS_\NN^c$.
We follow the proof of Theorem~\ref{thm:extra-abstract}, using Theorem~\ref{thm:N-factorization} for factorization and Lemma~\ref{lem:N-testing} for the H\"older type testing estimate. The factorized weights belong to $A_r(\NN)$ and are locally integrable. Hence $f$ belongs to their weighted $L_r$ spaces, and the hypothesis \eqref{eq:matrix-maximal-extrapolation-hypothesis} applies.

Suppose first that $p<r$. By Theorem~\ref{thm:N-factorization}, the weight produced by factorization satisfies
\[
 [|v^{-1/2}w^{1/2p}|^{2r}]_{A_r(\NN)}
 \le
 (4\WW_{p,d})^{r-p}
 [w]_{A_p(\NN)}^{(r-1)/(p-1)}.
\]
Applying \eqref{eq:matrix-maximal-extrapolation-hypothesis}
with this weight and completing the H\"older type testing argument using Lemma~\ref{lem:N-testing}(i), as in the proof of Theorem~\ref{thm:extra-abstract}, gives
\[
 \big\|{\sup_{\alpha\in\AA}}^+T_\alpha(f)
 \big\|_{L_p(\NN,w)}
 \le
 4\Phi\!\left(
   (4\WW_{p,d})^{r-p}
   [w]_{A_p(\NN)}^{(r-1)/(p-1)}
 \right)
 \|f\|_{L_p(\NN,w)}.
\]

Suppose now that $r<p$.
Theorem~\ref{thm:N-factorization} gives
\[
 [|v^{1/2}w^{1/2p}|^{2r}]_{A_r(\NN)}
 \le
 (4\WW_{p',d})^{(p-r)/(p-1)}
 [w]_{A_p(\NN)}.
\]
Applying
\eqref{eq:matrix-maximal-extrapolation-hypothesis}
with this weight, using Lemma~\ref{lem:N-testing}(ii) as in the proof of Theorem~\ref{thm:extra-abstract}, yields
\[
 \big\|{\sup_{\alpha\in\AA}}^+T_\alpha(f)
 \big\|_{L_p(\NN,w)}
 \le
 8\Phi\!\left(
   (4\WW_{p',d})^{(p-r)/(p-1)}
   [w]_{A_p(\NN)}
 \right)
 \|f\|_{L_p(\NN,w)}.
\]

Finally, $\SS_\NN^c$ is dense in $L_p(\NN,w)$ by
Lemma~\ref{lem:density of N}, so the map $f\mapsto(T_\alpha(f))_{\alpha\in\AA}$ extends uniquely to a bounded linear map on $L_p(\NN,w)$.

Thus we complete the proof by choosing the increasing function $\Phi_{p,r,d}$ as
    \begin{equation}\label{eq:matrix-Phi-prime}
    \Phi_{p,r,d}(t)=
    \begin{cases}
    4\Phi\!\left((4\WW_{p,d})^{r-p}t^{(r-1)/(p-1)}\right),
    &1<p\leq r,\\[0.5em]
    8\Phi\!\left((4\WW_{p',d})^{(p-r)/(p-1)}t\right),
    &1<r<p.
    \end{cases}
\end{equation}

\end{proof}

\begin{remark}\label{rem:finite-corner-alternative}
There is an alternative finite-corner proof.  One may work on
\[
 \NN_m=L_\infty(-m,m)\bar\otimes M_d(\C),
\]
extend a local matrix weight to the line by translation and reflection,
apply the trace-finite theorem with constants independent of $m$, and recover
the global norms by central localization.  The standard reflected extension
has a uniform but non-sharp quantitative loss in its $A_p$ characteristic.
The direct argument above is shorter and avoids that additional loss.
\end{remark}

\begin{remark}\label{rem:dimension-dependence}
The dependence on $d$ enters through the constants $\WW_{p,d}$ in the maximal characterization. Neither the factorization nor the direct passage to $\NN$ introduces any further dimensional loss.
\end{remark}

\subsection{Applications}

We apply Theorem~\ref{thm:intro-extrapolation} to the sharp
truncations of Calder\'on--Zygmund operators acting entrywise
on matrix-valued functions. The starting estimate at exponent $2$
follows from the classical Cotlar inequality and existing
matrix-weighted estimates.

Let $T$ be a linear operator bounded on $L_2(\R)$, with
kernel $K: \{(t,s)\in \R^2:t\neq s\}\rightarrow\mathbb C$ satisfying
\begin{equation}\label{eq:application-kernel}
\begin{aligned}
 |K(t,s)|
 &\le \frac{C_K}{|t-s|},\\
 |K(t,s)-K(t',s)|+|K(s,t)-K(s,t')|
 &\le C_K\frac{|t-t'|^\gamma}{|t-s|^{1+\gamma}},
 \qquad |t-t'|\le \frac{|t-s|}{2},
\end{aligned}
\end{equation}
where $C_K>0$ and $0<\gamma<1$. We assume the usual
off-support representation
\[
 Tf(t)=\int_{\R}K(t,s)f(s)\,ds,
 \qquad t\notin\operatorname{supp}f.
\]
For $\varepsilon>0$, define
\begin{equation}\label{eq:application-truncations}
 T_\varepsilon f(t)
 =\int_{|t-s|>\varepsilon}K(t,s)f(s)\,ds,
 \qquad f\in\SS_\NN^c.
\end{equation}

Recall the operator weight on $S_2^d$ from Section~\ref{sec:new Ap}:
\[
 \mathbb W(t)(A)=w(t)^{1/2}Aw(t)^{1/2},
 \qquad
 \mathbb W(t)^{1/2}(A)=w(t)^{1/4}Aw(t)^{1/4}.
\]
At exponent $2$, the identification $S_2^d\simeq\ell_2^{d^2}$ gives
\begin{equation}\label{eq:CZO-L2-identification}
 \|\mathbb W^{1/2}f\|_{L_2(\R;S_2^d)}
 =\|f\|_{L_2(\NN,w)}
\end{equation}
and
\begin{equation}\label{eq:CZO-characteristic-comparison}
 [\mathbb W]_{A_2(\R;M_{d^2}(\C))}
 =[w]_{A_2(\NN)}
\end{equation}
The Christ--Goldberg maximal operator at exponent $2$ is
\[
 M_{\mathbb W}g(t)
 =\sup_{I\ni t}\frac1{|I|}\int_I
 \|\mathbb W(t)^{1/2}\mathbb W(s)^{-1/2}g(s)\|_{S_2^d}\,ds.
\]

\begin{lemma}\label{lem:application-weighted-Cotlar}
Let $w\in A_2(\NN)$. Then
\begin{equation}\label{eq:application-maximal-L2}
 \big\|{\sup_{\varepsilon>0}}^+T_\varepsilon f
       \big\|_{L_2(\NN,w)}
 \le C_{d,\gamma}C_T[w]_{A_2(\NN)}^{5/2}
       \|f\|_{L_2(\NN,w)},\qquad f\in \SS_\NN^c.
\end{equation}
Here $C_T=\|T\|_{L_2(\R)\to L_2(\R)}+C_K$ and $C_{d,\gamma}$ is a constant depending on $d$ and $\gamma$.
\end{lemma}

\begin{proof}
Firstly, we show that
\begin{equation}\label{eq:weighted-Cotlar}
\begin{aligned}
 \big\|{\sup_{\varepsilon>0}}^+T_\varepsilon f
       \big\|_{L_2(\NN,w)}
 \le C'_{d,\gamma}[w]_{A_2(\NN)}
 \bigl(
   \|Tf\|_{L_2(\NN,w)}
   +C_T\|f\|_{L_2(\NN,w)}
 \bigr).
\end{aligned}
\end{equation}
The classical Cotlar inequality
\cite[(1) and Section~3.3, Theorem~5]{MOPV2010},
applied to $T/C_T$, gives
\begin{equation}\label{eq:application-scalar-Cotlar}
 \sup_{\varepsilon>0}|T_\varepsilon f(t)|
 \le C_\gamma M(|Tf|)(t)
       +C_\gamma C_T M(|f|)(t)
\end{equation}
for scalar inputs, where $M$ is the ordinary
Hardy--Littlewood maximal operator.
Now given $g\in \SS_\NN^c$ and $A\in M_d(\C)$, applying \eqref{eq:application-scalar-Cotlar} to the scalar-valued function $t\mapsto \tr(A^* g(t))$, we have
\[
\begin{aligned}
    \sup_{\varepsilon>0}|\tr(A^* T_\varepsilon g(t))|
 &\le C_\gamma M(|\tr(A^*(Tg))|)(t)
       +C_\gamma C_T M(|\tr(A^*g)|)(t) \\
       &\leq C_\gamma \| A\|_{S_2^d}M(\|Tg\|_{S_2^d})(t)
       +C_\gamma C_T \| A\|_{S_2^d}M(\|g\|_{S_2^d})(t).
\end{aligned}
\]
Taking the supremum over $\| A\|_{S_2^d}\leq 1$ gives
\[
\sup_{\varepsilon>0}\| T_\varepsilon g(t)\|_{S_2^d}
       \leq C_\gamma M(\|Tg\|_{S_2^d})(t)
       +C_\gamma C_T M(\|g\|_{S_2^d})(t).
\]
Fix $t\in\R$ and consider $g\colon u\mapsto w(t)^{1/4}f(u)w(t)^{1/4}$; then we have
\begin{equation}\label{eq:application-frozen-Cotlar}
    \sup_{\varepsilon>0}\| \mathbb W^{1/2}(t)T_\varepsilon f(t)\|_{S_2^d}
       \leq C_\gamma M_{\mathbb W}(\mathbb W^{1/2} Tf)(t)
       +C_\gamma C_T M_{\mathbb W}(\mathbb W^{1/2} f)(t).
\end{equation}
Lemma~\ref{lem:quantitative goldberg} gives Goldberg's quantitative maximal inequality
\begin{equation}\label{eq:application-CG-estimate}
 \|M_{\mathbb W}g\|_{L_2(\R)}
 \le \mathcal{G}_{2,d^2}[\mathbb W]_{A_2}
       \|g\|_{L_2(\R;S_2^d)}.
\end{equation}
Moreover, Lemma~\ref{lem:positive-max-to-classical} gives
\begin{equation}\label{eq:application-maximal-comparison}
 \big\|{\sup_{\varepsilon>0}}^+T_\varepsilon f
       \big\|_{L_2(\NN,w)}
 \le \sqrt d\left\|
   \sup_{\varepsilon>0}
   \|\mathbb W^{1/2}T_\varepsilon f\|_{S_2^d}
 \right\|_{L_2(\R)}.
\end{equation}
Thus combining \eqref{eq:CZO-L2-identification},
\eqref{eq:CZO-characteristic-comparison}, and
\eqref{eq:application-frozen-Cotlar}--\eqref{eq:application-maximal-comparison}
proves \eqref{eq:weighted-Cotlar} with constant $C'_{d,\gamma}=\sqrt{d}\cdot \mathcal{G}_{2,d^2}C_\gamma$.
Finally, the matrix $A_2$ theorem
\cite[Theorem~1.1 and Remark~1.3]{NPTV2017},
applied to $\mathbb W$ on $S_2^d$, gives
\begin{equation}\label{eq:application-single-L2}
 \|Tf\|_{L_2(\NN,w)}
 \le C''_{d,\gamma}C_T[w]_{A_2(\NN)}^{3/2}
       \|f\|_{L_2(\NN,w)}.
\end{equation}
Substituting \eqref{eq:application-single-L2} into
\eqref{eq:weighted-Cotlar}, and using
$[w]_{A_2(\NN)}\ge1$, yields
\[
\begin{aligned}
 \big\|{\sup_{\varepsilon>0}}^+T_\varepsilon f
       \big\|_{L_2(\NN,w)}
 &\le C'_{d,\gamma}C_T
 \bigl(C''_{d,\gamma}[w]_{A_2(\NN)}^{5/2}+[w]_{A_2(\NN)}\bigr)
 \|f\|_{L_2(\NN,w)}\\
 &\le C_{d,\gamma}C_T[w]_{A_2(\NN)}^{5/2}
       \|f\|_{L_2(\NN,w)},
\end{aligned}
\]
where $C_{d,\gamma}=C'_{d,\gamma}(C''_{d,\gamma}+1)$. This proves \eqref{eq:application-maximal-L2}.
\end{proof}

\begin{theorem}\label{thm:CZO-maximal}
\label{cor:CZO-matrix-Ap}
Let $T$ satisfy the preceding hypotheses.
For every $1<p<\infty$ and $w\in A_p(\NN)$, the map
\(
 f\mapsto(T_\varepsilon f)_{\varepsilon>0}
\)
extends uniquely from $\SS_\NN^c$ to a bounded map from
$L_p(\NN,w)$ into $L_p(\NN,w;\ell_\infty)$, and
\begin{equation}\label{eq:truncation-extrapolation}
 \big\|{\sup_{\varepsilon>0}}^+T_\varepsilon f
       \big\|_{L_p(\NN,w)}
 \le C_{p,d,\gamma}C_T
 [w]_{A_p(\NN)}^{\frac52\max\{1,\frac1{p-1}\}}
 \|f\|_{L_p(\NN,w)}.
\end{equation}
\end{theorem}

\begin{proof}
Lemma~\ref{lem:application-weighted-Cotlar} verifies the
hypothesis of Theorem~\ref{thm:intro-extrapolation}
for the linear family $(T_\varepsilon)_{\varepsilon>0}$,
with starting exponent $r=2$ and control function
\(
 \Phi(t)=C_{d,\gamma}C_Tt^{5/2}.
\)
The revised formula \eqref{eq:matrix-Phi-prime} gives
\[
\begin{aligned}
 \Phi_{p,2,d}(t)
 &=4C_{d,\gamma}C_T
   (4\WW_{p,d})^{\frac52(2-p)}
   t^{\frac{5}{2(p-1)}},
 &&1<p<2,\\
 \Phi_{p,2,d}(t)
 &=8C_{d,\gamma}C_T
   (4\WW_{p',d})^{\frac{5(p-2)}{2(p-1)}}
   t^{5/2},
 &&2<p<\infty.
\end{aligned}
\]
Together with \eqref{eq:application-maximal-L2}
at $p=2$, these estimates prove
\eqref{eq:truncation-extrapolation}.
The density of $\SS_\NN^c$ in $L_p(\NN,w)$ gives
uniqueness of the extension.
\end{proof}

\begin{remark}
The proof of Lemma~\ref{lem:application-weighted-Cotlar}
reduces the weighted $L_2$ problem to classical matrix-weighted
estimates on $S_2^d$ and then uses the finite-dimensional
comparison in Lemma~\ref{lem:positive-max-to-classical}.
This yields the dependence on $d$ of the bound in \eqref{eq:application-maximal-L2}. This raises the natural question whether the bound is sharp, which in turn motivates the question whether Lemma~\ref{lem:application-weighted-Cotlar}
admits a direct proof without reductions.

\end{remark}

\section{Open problems}\label{sec:problem}

We conclude with several questions about extensions of the symmetric
weighted realization and the limitations of the present method.

\subsection{Dimension dependence}
In Lemma~\ref{lem:N-Wp-symmetry}, the symmetry of the weight classes is obtained via
\[
 \gamma_p(t)=\gamma_{p,d}(t)=\WW_{p',d}t^p
\]
with $\WW_{p,d}\leq d^{\frac{1}{p}+(1+p')|\frac{1}{p}-\frac{1}{2}|}\mathcal{G}_{p,d^2}$.
The dependence on the matrix dimension $d$ raises the following questions.

\begin{enumerate}
\item
What is the sharp order of $\gamma_{p,d}$ with respect to $d$? Can we take it to be independent of $d$? A positive answer to the latter question would lead to a truly noncommutative weighted and extrapolation theory. 

\item A related question is the sharp order of the bound having appeared in \eqref{eq:truncation-extrapolation} with respect to $d$. Can it be taken to be independent of $d$?

\item Another related question is the existence of another generalization of classical $A_p$ weights beyond the one introduced by Treil and Volberg \cite{TreilVolberg1997,Volberg1997} such that the resulting bounds of the weighted Hardy-Littlewood maximal operator, the weighted Hilbert transform or Calder\'on-Zygmund operators are independent of $d$.
\end{enumerate}

\subsection{Other questions}

\begin{enumerate}
\item
In the abstract framework of Section \ref{sec:abstract Rubio}, we work with trace-finite von Neumann algebras because our use of Assumption \ref{ass:symmetric} involves \(w\), its inverse \(w^{-1}\), and products formed from them. In the trace-finite setting, given a positive full-support operator $w\in L_0(\M)$, all objects mentioned above belong to $ L_0(\M)$, which provides a common setting for our extrapolation theory.

For the particular semifinite algebra
\(
\mathcal N=L_\infty(\mathbb R)\bar\otimes M_d(\mathbb C),
\)
we instead work with strongly measurable matrix-valued functions. The inverse of a full-support function is defined pointwise almost everywhere and remains strongly measurable, as do the relevant products. Moreover, the geometry of \(\mathbb R\) provides a natural notion of local integrability, which allows us to develop the weighted theory in this setting.

For a general semifinite von Neumann algebra, however, a positive full-support operator \(w\in L_0(\mathcal M)\) need not satisfy \(w^{-1}\in L_0(\mathcal M)\). Nor is there, in general, an underlying geometric structure providing a corresponding notion of local integrability. How to formulate suitable weight classes and extend our extrapolation theory to this setting therefore remains open.

\item
In Section \ref{sec:abstract Rubio}, the factorization proof uses the positivity of the symmetric weighted realization
\(
 x\mapsto w^{1/2p}xw^{1/2p}.
\)
Can a similar extrapolation theorem be proved for
nonsymmetric realizations
\[
 x\longmapsto w^{\theta/p}xw^{(1-\theta)/p},
 \qquad
 0\le\theta\le1,\quad\theta\ne\frac12?
\]
A positive answer would provide an alternative proof of the extrapolation for Treil-Volberg's matrix-weighted theory. 

\item
In the classical setting, Rubio de Francia's extrapolation theorem applies to sublinear operators~\cite{Rubio1984}. This includes maximal operators and square functions, two central objects in harmonic analysis. In the noncommutative setting, however, these objects do not admit the same pointwise sublinear formulation.

Theorem~\ref{thm:intro-extrapolation} has established extrapolation for maximal estimates. It is therefore natural to ask whether an analogous result holds for square-function estimates. A basic complication is that in the noncommutative case, we may have to deal with the column and row square-function inequalities simultaneously:
\[
\begin{aligned}
 \Big\|
 \Big(
   \sum_{n}
   T_n(x)^*T_n(x)
 \Big)^{1/2}
 \Big\|_{L_p(\M,w)}
 &\le C\|x\|_{L_p(\M,w)}\\
 \Big\|
 \Big(
   \sum_{n}
   T_n(x)T_n(x)^*
 \Big)^{1/2}
 \Big\|_{L_p(\M,w)}
 &\le C\|x\|_{L_p(\M,w)}.
\end{aligned}
\]
\end{enumerate}

\bigskip

\noindent\textbf{Acknowledgements.}
The authors would like to thank Marius Junge, Javier Parcet and \'Eric Ricard for some discussions or suggestions in establishing the abstract factorization and extrapolation theorem in Section \ref{sec:abstract Rubio}.
This work was partially supported by the National Natural Science
Foundation of China (grant Nos.~12325105,
W2441002, and W2611005).

\smallskip

\noindent\textbf{AI statement.}
During the preparation of this manuscript, the authors used 
GPT only to assist with English-language
editing, the organization and presentation of the manuscript,
\LaTeX{} formatting, and preliminary checks of mathematical
consistency. All mathematical ideas, results, arguments, and proofs
were developed and independently verified by the authors. The authors
carefully reviewed and edited all AI-assisted suggestions and take full
responsibility for the accuracy, originality, and content of the
manuscript.

\providecommand{\bysame}{\leavevmode\hbox to3em{\hrulefill}\thinspace}
\providecommand{\MR}{\relax\ifhmode\unskip\space\fi MR }
\providecommand{\MRhref}[2]{\href{http://www.ams.org/mathscinet-getitem?mr=#1}{#2}
}
\providecommand{\href}[2]{#2}

\end{document}